\pdfoutput=1

\documentclass[12]{amsart}

\usepackage{amsmath}
\usepackage{amsfonts}
\usepackage{rotating}
\usepackage{bbm}
\usepackage[all,cmtip]{xy}
\usepackage{subcaption}
\usepackage{amssymb}
\usepackage[latin1]{inputenc}
\usepackage{graphicx}
\usepackage{dsfont}
\usepackage{color}
\usepackage{hyperref}

\renewcommand{\phi}{\varphi}

\newcommand{\R}{{\mathbb{R}}}

\newcommand{\Z}{{\mathbb{Z}}}

\advance\hoffset by -0.5 truecm
\usepackage{amssymb}
\newtheorem{Theorem}{Theorem}[section]
\newtheorem{Lemma}[Theorem]{Lemma}
\newtheorem{Corollary}[Theorem]{Corollary}

\newtheorem{Proposition}[Theorem]{Proposition}

\newtheorem{Definition}[Theorem]{Definition}
\newtheorem{Remark}[Theorem]{Remark}

\def\Z{{\mathbb Z}}
\def\R{{\mathbb R}}

\def \0{\lambda_{0}}

\begin{document}

\title[The Euler problem of two fixed centers]{Dynamical convexity of the Euler problem of two fixed centers}

\author{Seongchan Kim}
 \address{Universit\"at Augsburg, Universit\"atsstrasse 14, D-86159 Augsburg, Germany}
 \email {seongchan.kim@math.uni-augsburg.de}



\begin{abstract}
We give  thorough analysis for the rotation functions of the critical orbits from which one can understand bifurcations of periodic orbits. Moreover, we give explicit formulas of the Conley-Zehnder indices of the interior and exterior collision orbits and show that the universal cover of the regularized energy hypersurface of the Euler problem  is dynamically convex for energies below the critical Jacobi energy.
\end{abstract}

\maketitle
\tableofcontents

\section{Introduction}

The Euler problem of two fixed centers describes the motion of a massless body  under the influence of two fixed massive bodies according to Newton's law of gravitation. The two primaries will be referred to as the Earth and the Moon and  the massless body as  the satellite.  The problem  was first introduced by Euler in 1760 and he considered this problem as a starting point to study the restricted three body problem \cite{Euler1, Euler2}. Indeed, it can be obtained from the  planar  circular restricted three body problem by switching off the rotating term.  In 1902  Charlier  classified all orbits of this problem \cite[chapter 3]{Charlier}. He divided orbits into several classes according to the values of integrals.  Pauli \cite{Pauli}   studied this problem to investigate the hydrogen molecular ion ${H_2^+}$. In such a situation, the two primaries are regarded as two protons and the massless body as an electron.

The describing Hamiltonian   is given by 
\[
H : \big( \R^2 \setminus \left \{ \text{E}, \text{M} \right \} \big) \times \R^2 \rightarrow \R ,\;\;\;
(q,p)  \mapsto \frac{1}{2}|p|^2 - \frac{1-\mu}{ |q-\text{E}|} - \frac{\mu}{|q-\text{M}|} ,
\]
where $\mu \in (0,1)$ is the mass ratio of the two primaries,  $\text{E}=(-1/2,0)$ and $\text{M}=(1/2,0)$, i.e.,  the Earth is located at $(-1/2,0)$ and the Moon at $(1/2,0)$.  Notice that for a negative energy, the motion of the satellite is bounded. Without loss of generality we  assume that $ \mu \in (0, 1/2] $, i.e.,  the Earth is stronger.  The Hamiltonian has a unique critical point $L=(l,0,0,0)$, where $ l =  \frac{ 1-2\sqrt{ \mu(1-\mu)}}{2 (1-2\mu)}$  for  $\mu \neq 1/2 $ and $l= 0 $ for $ \mu = 1/2 $. Note that the projection of the critical point on the configuration space  lies on the line segment joining the Earth and the Moon.  The energy value ${ c_J :=H(L)= -1 -2 \sqrt{ \mu(1-\mu)} }$ is referred to as  the critical Jacobi energy. There are another two distinguished energy levels that we denote by $ c_e $ and $ c_h$ at which the Liouville foliation changes, see Remark \ref{energyfola}.

For an energy $c$, the Hill's region is defined by
\begin{equation*}
\mathcal{K}_c := \pi ( H^{-1}(c)) \subset \R^2 \setminus \left \{ \text{E, M} \right\},
\end{equation*}
where $\pi : \big( \R^2 \setminus \left \{ \text{E, M} \right \} \big) \times \R^2 \rightarrow \R^2 \setminus \left \{ \text{E, M} \right \} $ is the projection along $\R^2$. For  $c < c_J$, the  region ${\mathcal{K}_c}$ consists of two bounded connected components: one is a neighborhood of  the Earth and the other is a neighborhood  of  the Moon. We denote these components by ${ \mathcal{K}_c^E}$ and ${ \mathcal{K}_c^M}$, respectively. For ${ c_J < c < 0 }$, these two components become connected. Notice that  there is  no unbounded component for a negative energy, which is not the case in the rotating Kepler problem and the restricted three body problem.  For $c>0$, the Hill's region is the plane with the two punctures:  $ \R^2 \setminus \left \{ \text{E, M} \right \} $. In what follows, we  consider  only negative energies.

An interesting feature of the Euler  problem is the fact that the system is completely intergrable,  which was first discovered by Euler. More precisely, there exists a smooth function $G$, which is called the first integral of the system,  other than the Hamiltonian $H$ such that $dH$ and $dG$ are linearly independent almost everywhere and they are in involution: $\left \{ H, G\right\} =0$.  There are  distinguished periodic orbits: the interior collision orbit, the exterior collision orbit, the double-collision orbit, the hyperbolic orbit and the elliptic orbit, see Figure \ref{motions}. They are critical orbits, more precisely the derivatives of the Hamiltonian and the first integral are linearly dependent along these orbits and hence the corresponding leaves of the Liouville foliation are singular.

Since the Earth and the Moon are fixed, one can consider them as the foci of a set of ellipses and hyperbolas. Thus, one can introduce the elliptic coordinates  $( \xi, \eta)$. Introducing the elliptic coordinates, the system becomes separable and one can compute the periods ${\tau_{\xi}}$ and $\tau_{\eta}$ of $\xi$- and $\eta$-oscillations of each orbit. By the rotation number of an orbit, we mean the ratio $R=\tau_{\eta}/\tau_{\xi}$, which  depends only on the value $(G , H)=(g, c)$.

\begin{Proposition} \label{prop1} We denote by $R_{\text{int}}$,  $R_{\text{ext}} ^{\text{E}}$, $R_{\text{ext}}^{\text{M}}$, $R_{\text{dou}}$, $R_{\text{hyp}}$ and $R_{\text{ell}}$ the rotation functions of the interior collision orbit, the exterior collision orbits in the Earth and the Moon components, the double-collision orbit, the hyperbolic orbit and the elliptic orbit. Then  the following hold
\begin{eqnarray*}
&(\text{a})& \text{$R_{\text{int}}$ strictly increases from $1$ to $\infty$ as an energy increases from $-\infty$ to $c_J$.}  \\
&(\text{b})& \text{$R_{\text{ext}} ^{\text{E}}$  strictly increases from $1$ as an energy increases from $-\infty$ to $0$. }\\
&&\text{In particular, $\lim_{c \rightarrow 0} R_{\text{ext}} ^{\text{E}} = \infty$ for the symmetric case $\mu =1/2$.} \\
&(\text{c})&  \text{$R_{\text{ext}}^{\text{M}}$  strictly increases from $1$ to $\infty$ as an energy increases from $-\infty$ to $c_h$.}\\
&&\text{For $c \in ( c_h,0)$,  we have $R_{\text{ext}}^{\text{M}} = \infty$. }\\
&(\text{d})&  \text{$R_{\text{dou}}$  strictly decreases from $\infty$ to $0$ as an energy increases from $c_J$ to $c_e$.}\\
&&\text{For $ c \in (c_e, 0)$, we have $R_{\text{dou}} = 0$.} \\
&(\text{e})& \text{$R_{\text{hyp}}=\infty$. }\\
&(\text{f})&  \text{$R_{\text{ell}}$ strictly increases from $0$ to $1$ as an energy increases from $c_e$ to $0$}.
\end{eqnarray*}
\end{Proposition}

Throughout the paper, we fix some conventions. The symplectic form is given by $ \omega = \sum dp_j \wedge dq_j$ and the Hamiltonian vector field of a Hamiltonian $H$ is defined by $\omega( X_H, \cdot ) = -dH$.  In the following theorem, the Conley-Zehnder index is the one introduced by Hofer-Wysocki-Zehnder \cite[chapter 3]{HWZ}, which coincides with the transversal Conley-Zehnder index if a periodic orbit is nondegenerate.

\begin{Theorem}
\label{thm:formula} Assume that an energy is less than the critical Jacobi energy.  Then  the $2N$-th iteration of the interior collision orbit  is    nondegenerate if and only if   $2NR_{\text{int}} \notin \Z$ and its Conley-Zehnder index is given by
\begin{equation*}
\mu_{\text{CZ}}(\gamma^{2N}_{\text{int}}) = 1 + 2\max \left \{ k \in \Z : k < 2NR_{\text{int}} \right \} .
\end{equation*}

The $2N$-th iteration of the exterior collision orbit is   nondegenerate if and only if $2N/R_{\text{ext}} \notin \Z$. The Conley-Zehnder index is given by
 \begin{equation*}
\mu_{\text{CZ}}(\gamma_{\text{ext}}^{2N}) = 1+2\max \left \{k \in \Z : k < 2N/R_{\text{ext}} \right \} .
\end{equation*}

In particular, the doubly-covered interior collision orbit fails to be nondegenerate if and only if $R_{\text{int}} = k/2$, where $ k \geqslant 2$. For the energy $c$ at which $R_{\text{int}} \in ( (k-1)/2, k/2)$, its Conley-Zehnder index is given by $2k-1$. On the other hand, the doubly-covered exterior collision orbit is always nondegenerate and its Conley-Zehnder index equals 3.
\end{Theorem}

As a corollary of Theorem \ref{thm:formula}, we obtain the following.

\begin{Theorem}
\label{thm:index}
For energies below the critical Jacobi energy, the universal cover of the regularized energy hypersurface is dynamically convex, namely all  periodic Reeb orbits  have Conley-Zehnder indices at least 3. 
\end{Theorem}

\;

 The concept of dynamical convexity is introduced by Hofer-Wysocki-Zehnder  \cite[definition 3.6]{HWZ}.  While  convexity  is not preserved under symplectomorphisms, dynamical convexity is a symplectic invariant. By the results of \cite{HWZ}  if $S^3$ is equipped with a dynamically convex contact form, then it admits an open book decomposition whose pages are (disk-like) global surfaces of section for the associated Reeb flow. By a  global surface of section we mean an embedded disk $D \subset S^3$ having the following properties:

(i) the boundary $\partial D$ is a periodic Reeb orbit, which is called the spanning orbit,

(ii) the Reeb vector field is transverse to the interior of $D$,

(iii) every orbit, other than the spanning orbit, intersects the interior of $D$ in forward and backward time.\\
Note that all the global surfaces of section above are spanned by the binding of the open book decomposition. This binding  is a nondegenerate  periodic orbit which is unknotted,  has self-linking number -1 and  is of Conley-Zehnder index 3. 

\;\;

By varying the energy level $c$  one can construct a homotopy of the exterior collision orbits. In particular, as $c \rightarrow -\infty$ the exterior collision orbits are getting close to a simply-covered geodesic on the two-sphere \cite[section 2]{Moser}. Since via Levi-Civita regularization \cite{Levi} geodesic flows on the two-sphere lift to Hopf links, their double covers have self-linking number -1. Since self-linking number is a homotopy invariant, we conclude that self-linking number of the doubly-covered exterior collision orbits equals -1. Then the assertion of Theorem \ref{thm:index} implies that for each $c<c_J$, the universal cover of each component of the regularized energy hypersurface admits an open book decomposition whose binding is the doubly-covered exterior collision orbit. In particular, each doubly-covered exterior collision orbit bounds a global surface of section of the Reeb flow. In fact, in view of \cite[theorem 1.7]{Umberto}, the same argument also holds true for the doubly-covered interior collision orbits.

\begin{Remark} \rm We remark that the roles of the interior and the exterior collision orbits are reminiscent respectively of the roles of the direct and the retrograde circular orbits in the rotating Kepler problem, see \cite{RKP}.
\end{Remark}

\;\;

\textbf{Acknowledgments.} First and foremost, I would like to express my gratitude to  advisor Urs Frauenfelder for   encouragement and support as well as interesting me in this subject. I also thank to Yehyun Kwon and  Junyoung Lee for fruitful discussions and the unknown referee whose comments helped me give the arguments followed by Theorem \ref{thm:index}.  Furthermore, I want to thank to the Institute for Mathematics of University of Augsburg for providing a supportive research environment. This research was supported by DFG grants CI 45/8-1 and FR 2637/2-1.

\;\;

\section{The   Euler problem of two fixed centers}
\label{section:Euler}
 
As mentioned in the introduction, we introduce the elliptic coordinates which  are  defined by
\begin{eqnarray*}
\xi=  |q- \text{E} | + |q- \text{M}| \in [1,\infty) \;\;\;\text{ and  }\;\;\;\eta=   |q- \text{E}| - |q- \text{M}| \in [-1,1] .
\end{eqnarray*}
In the $(q_1, q_2)$-plane, $\xi =$ const or $\eta = $ const represents an ellipse or a hyperbola, respectively.   The corresponding momenta ${p_{\xi} }$ and ${ p_{\eta}}$ are determined by the relation $ p_1 d q_1 + p_2 dq_2 = p_{\xi} d \xi + p_{\eta} d \eta$  and then the Hamiltonian in the elliptic coordinates is of the form
\begin{equation*}
H = \frac{ H_{\xi} + H_{\eta}    }{\xi^2 - \eta^2 },
\end{equation*}
where $ H_{\xi} = 2(\xi^2 -1)p_{\xi}^2   - 2 \xi$ and $H_{\eta} = 2 (1-\eta^2)p_{\eta}^2  + 2(1-2\mu)\eta$. Following the convention of Strand-Reinhardt \cite{Strand} we choose the first integral $G$  by
\begin{equation*}
G = - \frac{ \eta^2 H_{\xi} + \xi^2 H_{\eta} }{\xi^2 - \eta^2 } .
\end{equation*}
Given ${ (G,H) = (g,c) }$,    the momentum variables are given by
\begin{equation}
\label{eq:momentum}
p_{\xi}^2 = \frac{ c\xi^2 + 2\xi +g}{ 2 ( \xi^2 -1)} \;\;\;\text{ and  }\;\;\; p_{\eta}^2 = \frac{  c\eta^2 +2(1-2\mu)\eta + g}{2 ( \eta^2-1)}.
\end{equation}
We  now define two functions
\begin{equation}
\label{eq:functions}
f(\xi) = ( c\xi^2 + 2 \xi + g )(\xi^2 -1) \;\;\;\text{ and  }\;\;\;h(\eta) = ( c\eta^2 + 2(1-2\mu)\eta + g )(  \eta^2 -1 ).
\end{equation}
The function ${f}$ (or ${h}$) has four roots: ${\pm 1}$ and 
$${ \xi_{1,2} = \frac{ -1 \pm \sqrt{ 1 - gc}}{c}} \;\; \bigg(\text{or} \;\;{ \eta_{1,2} = \frac{ -(1-2\mu) \pm \sqrt{ (1-2\mu)^2 - gc}}{c}}\bigg).$$
According to ranges of $\xi_{1,2}$ and $\eta_{1,2}$, the classically accessible regions in the lower-half $(g,c)$-plane are divided into four($\mu \neq 1/2$) regions or three($\mu = 1/2$) regions, see Figure \ref{fig:region}.  For details to obtain those regions, see for example \cite{Strand, Bifurcation, Kim}. Following the notations from \cite{Charlier, Pauli}, the regions are labeled by ${S'}$, ${S}$(satellite), ${L}$(lemniscate), and ${P}$(planetary).  In   ${S'}$, the satellite is confined to $\mathcal{K}_c^E$ while it also can move in  $\mathcal{K}_c^M$ in ${S}$. In  ${L}$, the movement is bounded by the ellipse ${ \xi = \xi_1 }$. Finally, the satellite moves between two ellipses ${ \xi = \xi_1 }$ and $\xi = \xi_2$ in ${P}$, see Figure \ref{motionsd}. These regions are bounded by the five critical curves
\begin{eqnarray*} 
&&l_{1,2} : c=-g\pm2(1-2\mu),\;\;\; \;\;\; \;\;\; \;\;\; \;\;\; l_3 : c=-g-2,\\
&& l_4 : gc=(1-2\mu)^2,~c_J<c<c_h,\;\;\; \; l_5 : gc=1,~c_e<c.
\end{eqnarray*} 
All points on these five curves are  critical values of the energy momentum mapping ${(\xi,\eta) \mapsto ( G(\xi, \eta), H(\xi, \eta) ) }$, while  points  in the interior of each region  are  regular values, namely  they represent Liouville tori.   For the symmetric case, the lines ${l_1}$ and $l_2$ are identical and  the region ${S'}$ does not appear. The ranges of ${ \xi}$ and ${\eta}$ for the motions in each region are presented in Table \ref{table1}. 

\begin{figure}[h]
 \centering
 \includegraphics[width=0.5\textwidth, clip]{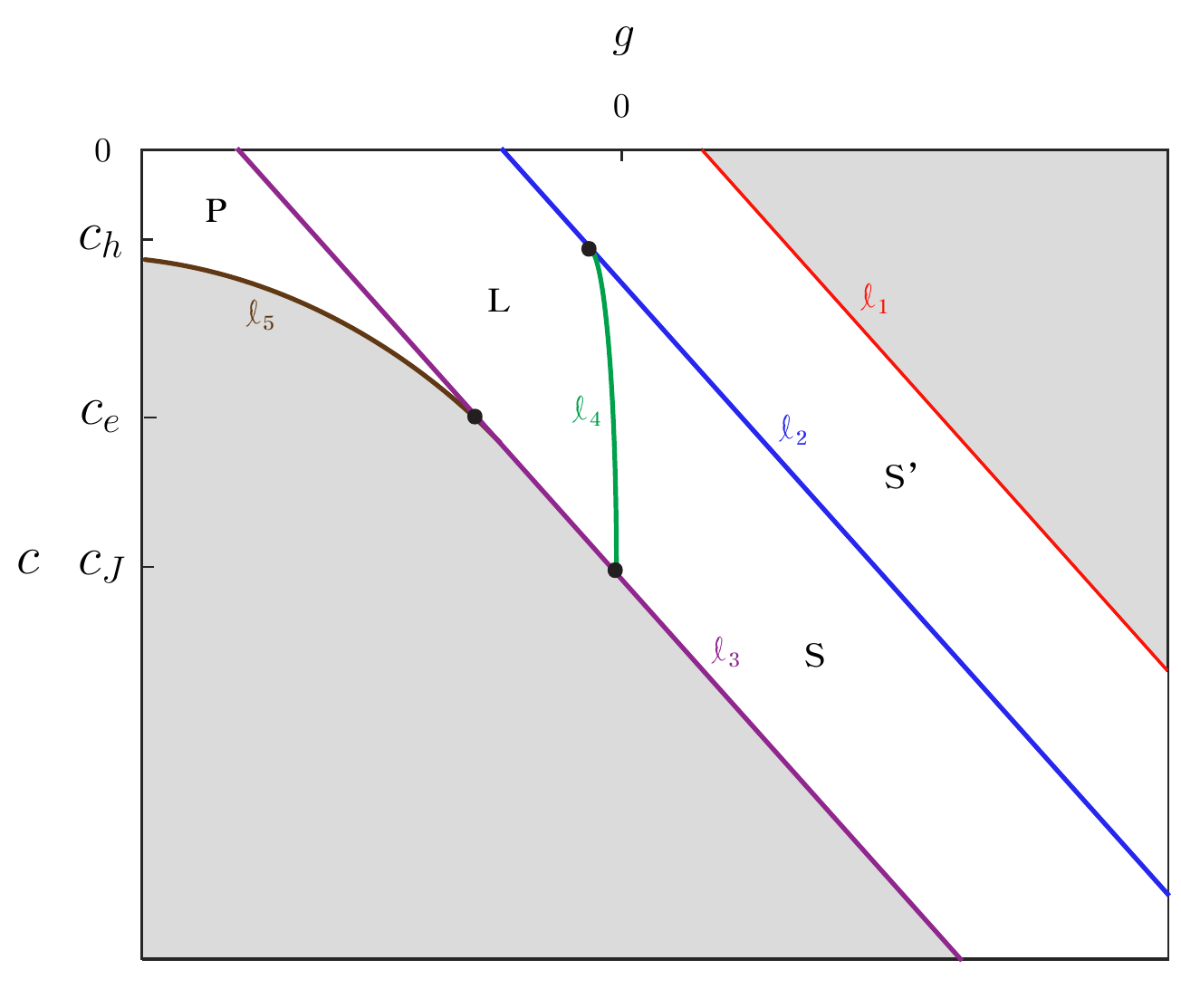}
 \caption{ For negative energies, each regular level of the energy-momentum mapping represents one of the four types of motions, which are labeled by $S'$, $S$, $L$ and $P$. The colored curves are the critical curves which divide the four regular regions. The shaded regions are classically forbidden. }
 \label{fig:region}
\end{figure}
\begin{figure}[h]
\begin{subfigure}{0.4\textwidth}
  \centering
  \includegraphics[width=0.9\linewidth]{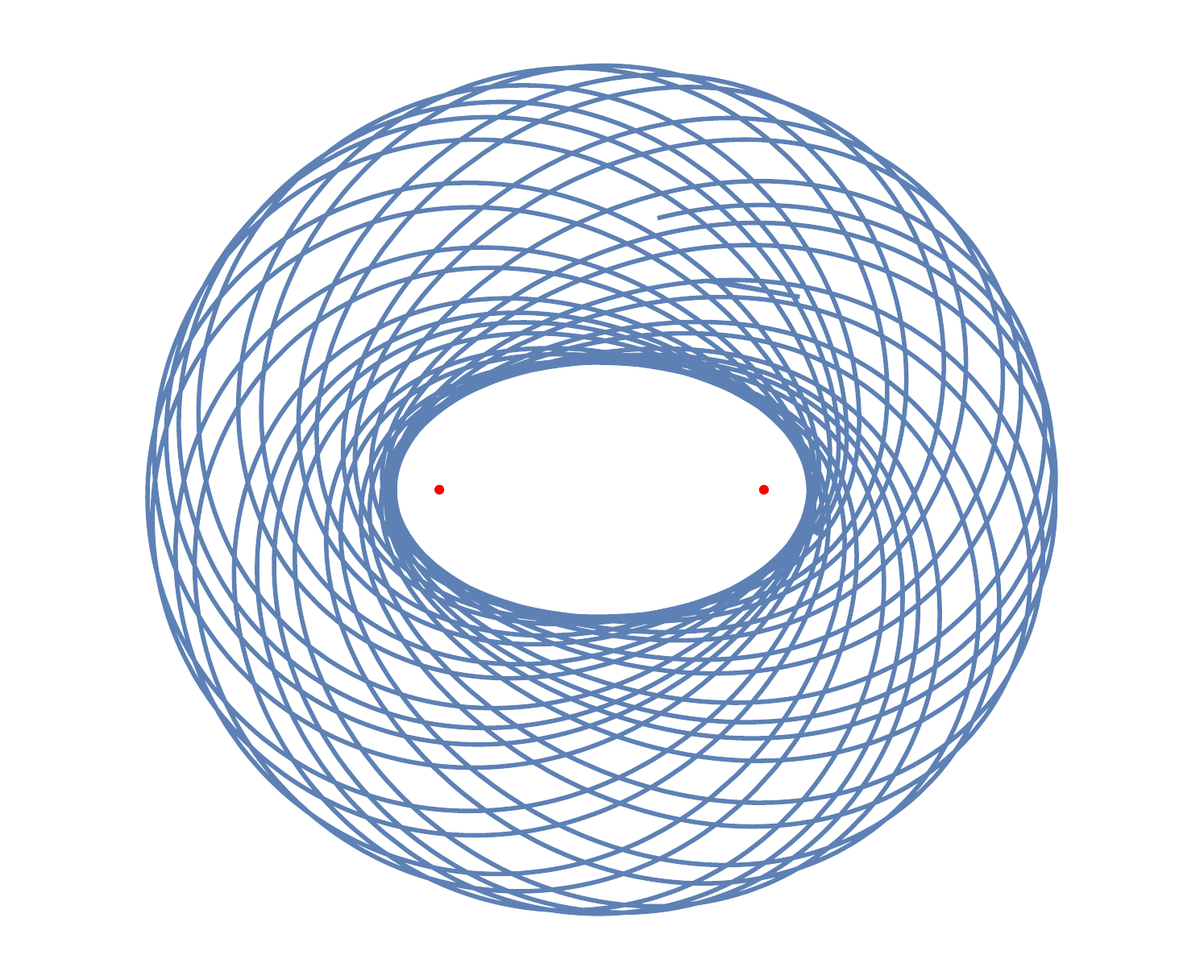}
  \caption{$P$-region}
\end{subfigure}
\begin{subfigure}{0.4\textwidth}
  \centering
  \includegraphics[width=0.9\linewidth]{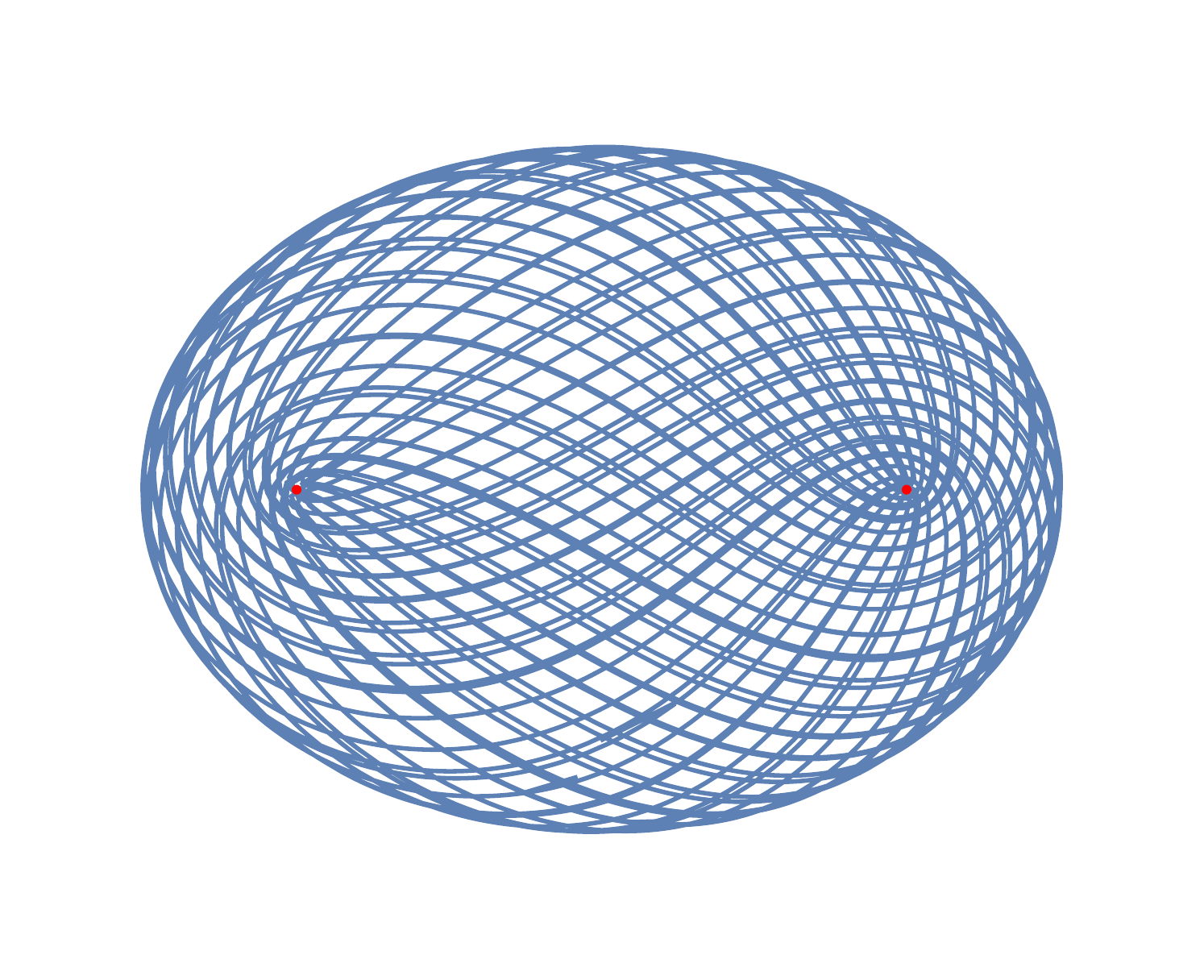}
  \caption{$L$-region}
\end{subfigure}
\begin{subfigure}{0.4\textwidth}
  \centering
  \includegraphics[width=0.9\linewidth]{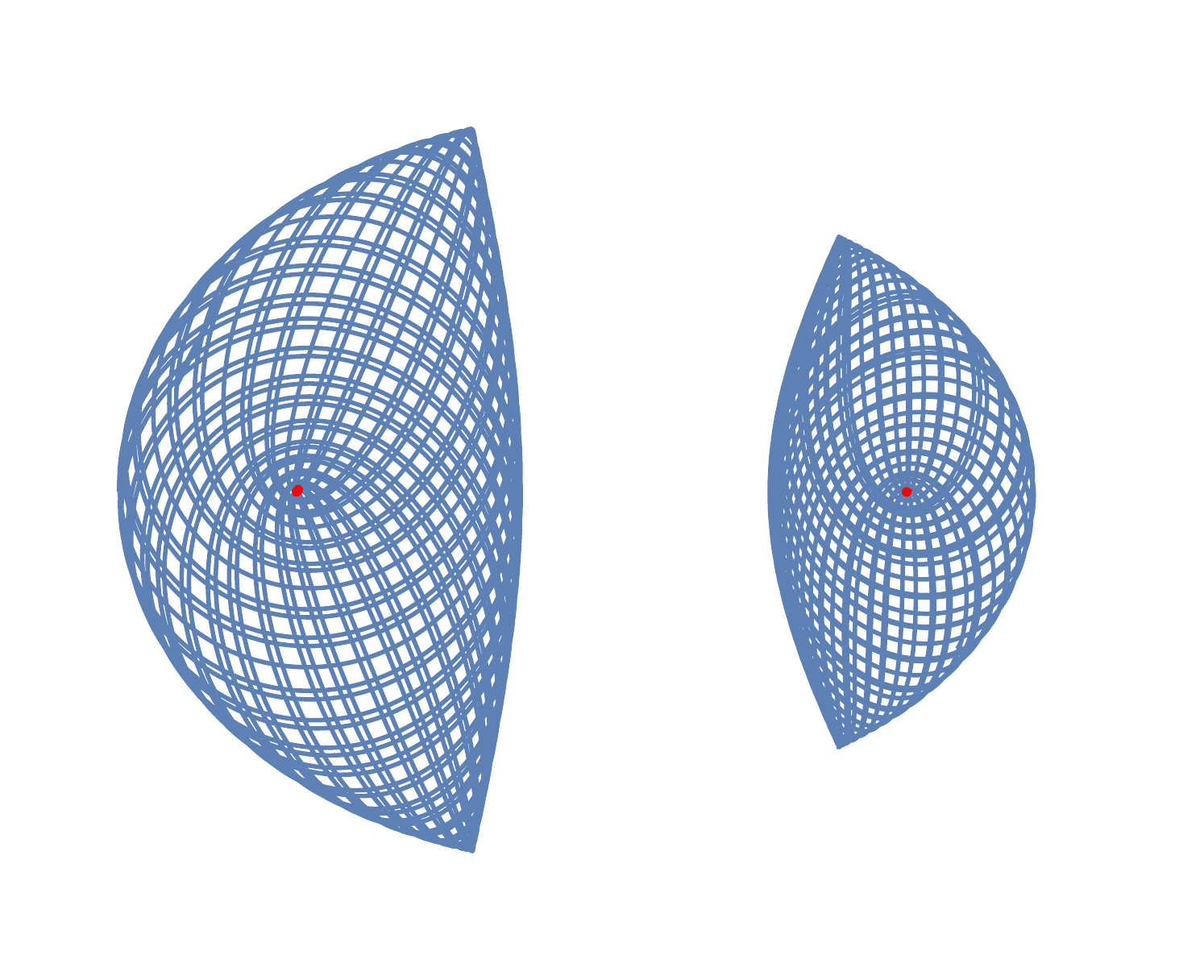}
  \caption{$S$-region}
\end{subfigure}
\begin{subfigure}{0.4\textwidth}
  \centering
  \includegraphics[width=0.9\linewidth]{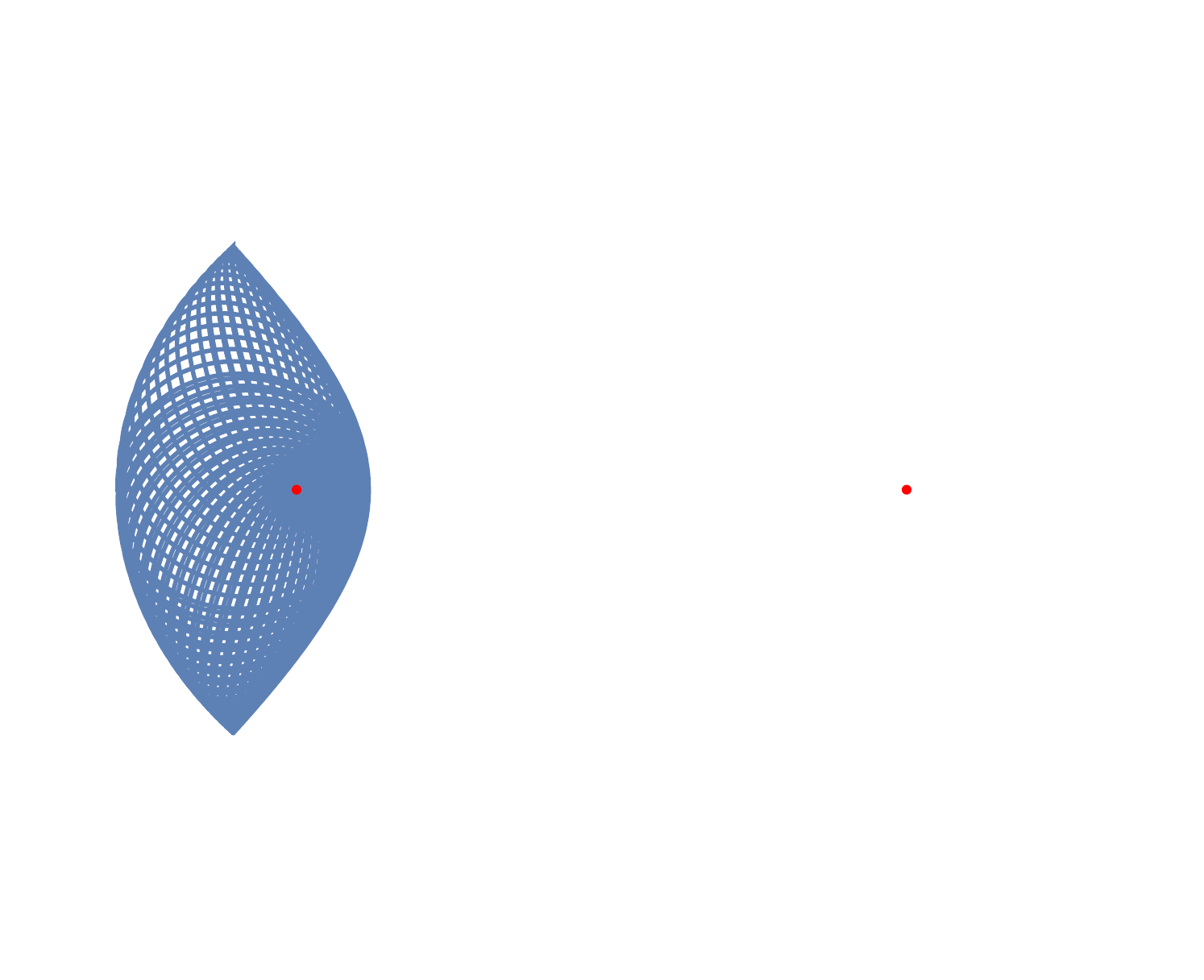}
  \caption{$S'$-region}
\end{subfigure}
\caption{  Typical orbits in the regular regions for $\mu = 1/4$  }
\label{motionsd}
\end{figure}

\begin{table}[t]  
\centering
\caption{ The ranges of the variables $\xi$ and $\eta$ in each regular region }
\begin{tabular}{llll}
Regions  & ${\xi}$-range & ${\eta}$-range & Ranges of roots \\
\hline
S$'$  &  (1, ${\xi_2 }$) & (${-1}$, ${\eta_1}$) & ${  -1<\xi_1 < 1 < \xi_2,~-1 < \eta_1 < 1 < \eta_2}$  \\
S, Earth   &  (1, ${\xi_2 }$) & (${-1}$, ${\eta_1}$)  & ${  -1<\xi_1 < 1 < \xi_2,~-1 < \eta_1 < \eta_2 < 1}$    \\ 
S, Moon &  (1, ${\xi_2 }$) & (${\eta_2}$,  1)  & ${  -1<\xi_1 < 1 < \xi_2,~-1 < \eta_1 < \eta_2 < 1}$  \\ 
L  &  (1, ${\xi_2 }$) & (${-1}$, 1)  &   ${ -1<\xi_1 < 1 < \xi_2,~-1  < 1 < \eta_1<\eta_2}$  if $(1-2\mu)^2 \geqslant gc$ \\
~& ~& ~&\;\;\;\;\; \;\;\;\;\;\;\;\;\;\;\;\; \;\;\;\;\;\; \; ${\text{ or }\eta_1, \eta_2 \text{ complex} }$ \; if $(1-2\mu)^2 < gc$ \\
P  &  (${\xi_1}$, ${\xi_2 }$) & (${-1}$, 1) &  ${ -1<1<\xi_1   < \xi_2 , ~-1  < 1 < \eta_1<\eta_2  }$  if $(1-2\mu)^2 \geqslant gc$ \\
~& ~& ~&\;\;\;\;\; \;\;\;\;\;\;\;\;\;\;\;\; \;\;\;\;\;\; \; ${\text{ or }\eta_1, \eta_2 \text{ complex} }$ \; if $(1-2\mu)^2 < gc$ \\
\end{tabular}
\label{table1}
\end{table}

\begin{Remark}\label{energyfola} \rm There are three specific energy values: the critical Jacobi energy  ${ c_J}$, ${c_e = -1}$ at which   ${ l_3}$ and $l_5$ intersect, and ${c_h = -1+2\mu}$ at which  $ l_2$ and $l_4$ meet. These distinguished  values are the energy levels at which the Liouville foliation on an energy hypersurface changes, see Figure \ref{fig:region}.
\end{Remark}

We now investigate the critical orbits. Every point on the line $l_1$ represents  the collision orbit ${\eta=-1 }$ in ${ \mathcal{K}_c^E}$. We call this orbit \textit{the exterior collision orbit} in the Earth component. A point on the line ${l_2}$ represents an orbit  in either ${\mathcal{K}_c^E}$ or      ${\mathcal{K}_c^M}$. An orbit in the Earth component is not a critical orbit. For an orbit near the Moon, it represents  the exterior collision orbit $\eta=1$ in   $\mathcal{K}_c^M$.  On the line ${l_3}$, for $c<c_J$, a point represents the collision orbit $\xi=1$ in either  $\mathcal{K}_c^E$ or   $\mathcal{K}_c^M$. Such an  orbit will be referred   to as \textit{the interior collision orbit}.  For ${ c > c_J}$, two interior collision orbits become connected and the satellite moves between two primaries, see Figure \ref{connectedinterior}. We call this orbit \textit{the double-collision orbit}.  For a point on the curve ${ l_4}$, the equation $c \eta^2 + 2 ( 1-2\mu) \eta + g = 0 $ has the common root which is positive: $  \eta_1 = \eta_2 = - ( 1-2\mu)/c > 0 $. This implies that the satellite moves along the hyperbola $ \eta = \eta_1 $, which is close to the Moon,  within the boundary ellipse $ \xi = \xi_1 $. We call this orbit \textit{the hyperbolic  orbit}.  Finally, on the curve ${ l_5}$  we have $ \xi_1 = \xi_2 $ and hence the orbit is an ellipse, which we call \textit{the elliptic  orbit}.

\begin{Remark} \rm The hyperbolic orbits are Lyapunov orbits, i.e., as $c \rightarrow c_J$ from above the family $\gamma_{\text{hyp}}^c$, $c \in (c_J, c_h)$, of the hyperbolic orbits converges uniformly to the critical point $L$ \cite[section 3]{Kim}. On the other hand,  the hyperbolic orbit degenerates to the exterior collision orbit in the Moon component as $c \rightarrow c_h$. Moreover, the elliptic orbit degenerates to the double-collision orbit as $c \rightarrow c_e$, see Figure \ref{hyperdeg}, \ref{ellipticdeg}.
\end{Remark}

\section{Rotation functions}

Note that the Hamiltonian has singularities at the Earth and the Moon and hence an energy hypersurface is not compact due to collisions. However, one can regularize this two-body collision by means of a suitable time rescaling as follows. 

Fix $c<0$. We now show how to regularize the dynamics on the energy level set $H^{-1}(c)$. Define the new Hamiltonian
\begin{equation*}
K := ( \xi^2 - \eta^2) (H-c) = H_{\xi} + H_{\eta} -c(\xi^2 -\eta^2).
\end{equation*}
For points at $K=0$ we have $\partial_{\sigma} K = ( \xi^2 - \eta^2) \partial_{\sigma} H$ for each $\sigma = \xi, \; \eta, \; p_{\xi},$ or $p_{\eta}$. Therefore, with the time scaling
\begin{equation*}
d t = ( \xi^2 - \eta^2 ) d \tau ,
\end{equation*}
orbits of $H$ with energy $c$ and time parameter $t$ correspond to orbits of $K$ with energy $0$ and time parameter $\tau$. Note that the energy hypersurface $H^{-1}(c)$  is compactified to $K^{-1}(0)$.  The equations of the momenta (\ref{eq:momentum}) and the Hamiltonian equations of $K$ on $K^{-1}(0)$ give rise to
\begin{equation*}
\begin{cases} \dot{\xi} = 4(\xi^2 -1)p_{\xi} = 2\sqrt{2}\sqrt{f(\xi)} \\  \dot{\eta} = 4(1- \eta^2 )p_{\eta} = 2\sqrt{2}\sqrt{h(\eta)} , \end{cases} 
\end{equation*}
where the dot denotes the differentiation with respect to $\tau$ and the functions $f$ and $h$ are defined as in (\ref{eq:functions}). It follows that 

\begin{eqnarray}
\label{eq:integral}
\int_{\xi_0}^{\xi(\tau)} \frac{ d \xi}{ 2 \sqrt{2}\sqrt{ f(\xi)}} = \tau - \tau_0 = \int_{\eta_0}^{\eta(\tau)} \frac{ d \eta}{2 \sqrt{2}\sqrt{ h(\eta)} }.
\end{eqnarray}

\begin{figure}[t]
\begin{subfigure}{0.9\textwidth}
  \centering
  \includegraphics[width=0.95\linewidth]{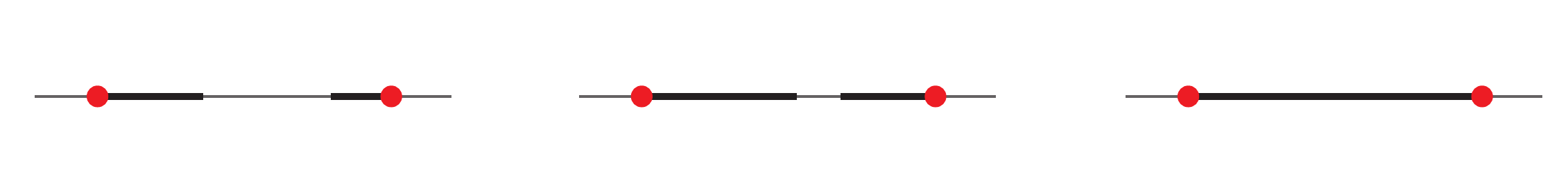}
  \caption{   }
  \label{connectedinterior}
\end{subfigure}
\begin{subfigure}{0.45\textwidth}
  \centering
  \includegraphics[width=0.9\linewidth]{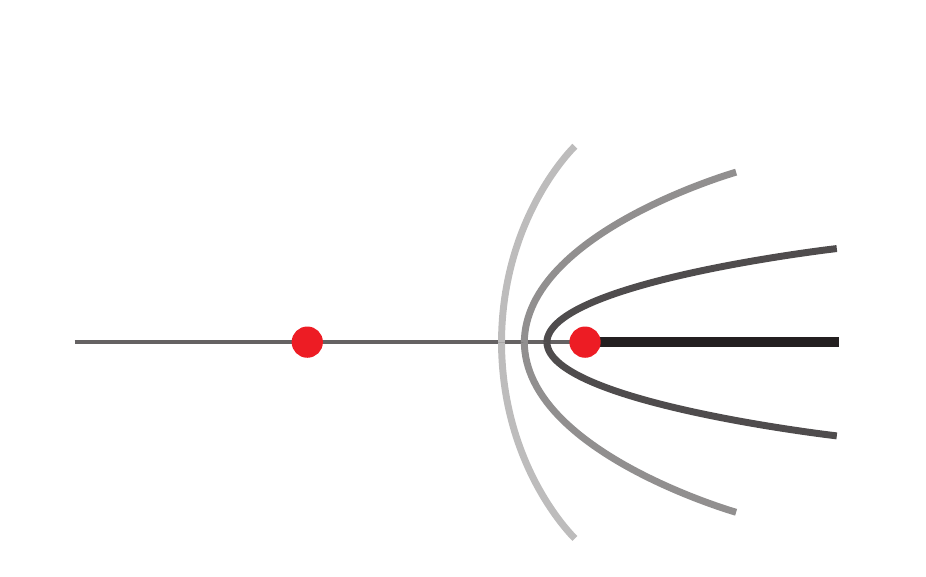}
  \caption{  }
 \label{hyperdeg}
\end{subfigure}
\begin{subfigure}{0.45\textwidth}
  \centering
  \includegraphics[width=0.9\linewidth]{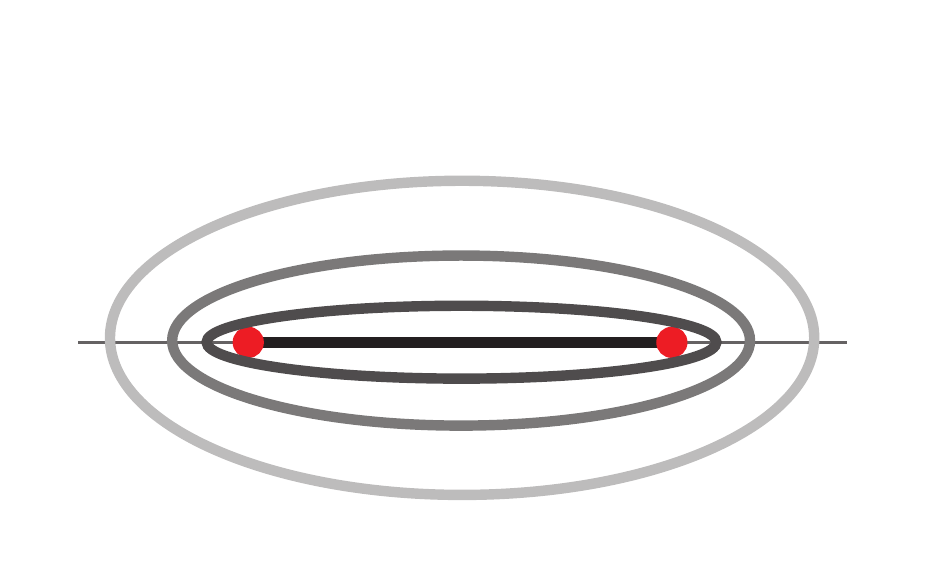}
  \caption{ }
 \label{ellipticdeg}
\end{subfigure}
\caption{  (a) The interior collision orbits become connected. (b) The hyperbolic orbit degenerates to the exterior collision orbit in the Moon component. (c)   The elliptic orbit degenerates to the double-collision orbit. }
\label{motions}
\end{figure}

\begin{Lemma}  
\label{lemma:periods}
The integrals (\ref{eq:integral}) over $[\sigma_{\text{min}}, \sigma_{\text{max}}]$, ${\sigma=\xi}$ or $\eta$, are given as follows.
\begin{eqnarray}
\label{integralxi}
\int_{\xi_{\text{min}}}^{\xi_{\text{max}}} \frac{d\xi}{ \sqrt{f(\xi)} } = \begin{cases} \displaystyle \frac{ 1} { \sqrt[4]{1 - gc} }  K(k_1)  &  \;\;\;\;\;\;\;\;\;\;\;\;\text{$S'$, $S$ and $L$-regions} \;\;\;\;\;\;\;\;\; \\  & \\ \displaystyle   \frac{ 2}{ \sqrt{ -g +c +  2\sqrt{ 1-gc}      }}    K(k_2) & \;\;\;\; \;\;\;\;\;\;\;\;\;\text{$P$-region} \end{cases}
\end{eqnarray} 
\begin{eqnarray}
\label{integraleta}
\;\;\;\;\;\;\;\; \int_{\eta_{\text{min}}}^{\eta_{\text{max}}} \frac{d\eta}{ \sqrt{h(\eta)} } = \begin{cases} \displaystyle \frac{1}{   \sqrt[4]{ (1-2\mu)^2 -gc}}K(r_1) & \text{$S'$-region} \\ & \\ \displaystyle      \frac{ 2}{\sqrt{ g-c+2\sqrt{ (1-2\mu)^2- gc}}}K(r_2) & \text{$S$-region} \\ & \\ \displaystyle     \frac{ 2}{\sqrt{-g +c +2 \sqrt{  ( 1-2\mu)^2 -gc }}}   K(r_3) & \text{$L$,  $P$-regions, $(1-2\mu)^2 \geqslant gc$} \\ & \\ \displaystyle \frac{ 2}{\sqrt[4]{ (g+c)^2 - 4(1-2\mu)^2}} K(r_4) & \text{$L$,  $P$-regions, $(1-2\mu)^2 < gc$}\end{cases}
\end{eqnarray}
where

\begin{eqnarray*}
&& k_1^2 =    \frac{1}{2} \begin{pmatrix} \displaystyle 1 - \frac{g-c}{2 \sqrt{1-gc}} \end{pmatrix}, \;\;\;\;\;\;\;\;\;\;\;\;\;\;\;\;\;\; k_2^2  =  \frac{ 1}{k_1^2} \\
&& r_1^2  =   \frac{1}{2} \begin{pmatrix} \displaystyle     1 - \frac{ g-c}{2\sqrt{ (1-2\mu)^2 -gc}}      \end{pmatrix} , \;\;\;\;\;  r_2^2 =  \frac{ g-c- 2\sqrt{ (1-2\mu)^2 -gc}}{ g-c+2\sqrt{ (1-2\mu)^2 -gc}}, \\  
&&  r_3^2  =     \frac{ 1}{r_1^2},  \;\;\;\;\; \;\;\;\;\;\; \;\;\;\;\;\; \;\;\; \;\;\; \;\;\; \;\;\; \;\;\; \;\;\; \;\;\; \;\;\;\;\;\; \; r_4^2 = \frac{ 1}{2}\bigg( 1  + \frac{ g-c}{ \sqrt{ (g+c)^2 - 4(1-2\mu)^2}}\bigg)
\end{eqnarray*}
and $K(k)$ is the complete elliptic integral of the first kind of modulus $k$, which can be expressed by the power series
\begin{equation}\label{ellipticequationpower}
K(k) = \frac{\pi}{2} \sum_{n=0}^{\infty} \bigg( \frac{ (2n-1)!!}{(2n)!!} \bigg)^2 k^{2n}.
\end{equation}
\end{Lemma}
\begin{proof}
See \cite{handbook, Demin}.
\end{proof}

\begin{Remark} \rm Since the absolute values of all the moduli in the lemma are less than 1, the corresponding power series (\ref{ellipticequationpower}) are convergent.
\end{Remark}

\begin{Remark}\label{ramket} \rm Since the $\eta$-motion in the Earth component is smooth,    the two ${\eta}$-integrals (\ref{integraleta}) for $S'$- and $S$-regions coincide on the line ${ l_2 : g = -c -2(1-2\mu)}$. Indeed, we  observe that on  $l_2$
\begin{eqnarray*}
 g-c&=& -2c -2(1-2\mu) = 2 ( -c-1+2\mu)\\
 \sqrt{ (1-2\mu)^2-gc}  &=&   | c+1-2\mu| =    - c -   1+2\mu  .
\end{eqnarray*}
It follows  that the first two integrals in (\ref{integraleta}) coincide with ${\pi}/{ ( 2 \sqrt{-c-1+2\mu} ) }.$ On the other hand, for the regions $L$ and $P$, the last two integrals in (\ref{integraleta}) are also identical for points $(g,c)$ satisfying $gc= (1-2\mu)^2$ with $c>c_h$. Indeed, if $gc=(1-2\mu)^2$, then we obtain $ 2 K(0) / \sqrt{ -g+c} = \pi / \sqrt{ -g+c}$.
\end{Remark}

We abbreviate by $\sqrt{2} \tau_{\xi}$ and $\sqrt{2} \tau_{\eta}$ the $\xi$- and $\eta$-integrals in the previous lemma. For each orbit, $\tau_{\xi}  $ and $\tau_{\eta} $ give rise to the periods of $\xi$- and $\eta$-oscillations.     A direct computation shows the  following lemma, where detailed computation will be included in Appendix \ref{appendix:lemma}.
\begin{Lemma}  
\label{lemma:periodsmonotone}
\rm{(Dullin-Montgomery, \cite[section 7]{syzygy})} Fix an energy $c$. Then in the regions $S'$ and $S$,  the periods $\tau_{\xi}$ and $\tau_{\eta}$ are decreasing in $g$  while both they are increasing in the region $P$. In the region $L$, the period $\tau_{\xi}$ is decreasing, but   $\tau_{\eta}$ is increasing. 
\end{Lemma}

\begin{Definition} \rm Given a regular level $(g,c)$ of the energy momentum mapping we define {the rotation number} of the corresponding Liouville torus by the ratio  ${R = \big(\tau_{\eta} / \tau_{\xi}\big)(g,c) }$. Fixing an energy level $c$ and varying the intergal $g$ defines {the rotation function} $R_c := R( \cdot, c)$ on the $c$-energy hypersurface.
\end{Definition}

We also define the rotation number of an orbit by the same formula. Note that all the orbits on the same torus have the same rotation number since it depends only on the value $(g,c)$. An orbit is periodic if and only if the rotation number is rational. For this reason, a Liouville torus on which periodic orbits lie is called a {rational torus}. On the other hand, given a rotation number $R_0$ the equation $ R(g,c)=R_0$ defines a smooth family of Liouville tori of the fixed rotation number $R_0$. If $R = k/l$, where $k,l \in \Z$ are relatively prime, then the corresponding family of Liouville tori is referred to as {the $T_{k,l}$-torus family}. Each torus  family draws a smooth curve in the lower-half $(g,c)$-plane. By definition we have 
\begin{Lemma}  
\label{lemma:rotation}
The rotation function for each region is given as follows.
\begin{eqnarray*}
&R_{S'}&  = \sqrt[4]{\frac{ 1-gc}{(1-2\mu)^2-gc}}\frac{ K(r_1)}{K(k_1)} \\
&R_S&  =  2 \sqrt{ \frac{ \sqrt{1-gc}}{ g-c+2\sqrt{ (1-2\mu)^2-gc}} } \frac{ K(r_2)}{K(k_1)} \\ 
&R_L& = \begin{cases} \displaystyle 2 \sqrt{ \frac{ \sqrt{1-gc}}{ -g+c+2\sqrt{ (1-2\mu)^2-gc}} } \frac{ K(r_3)}{K(k_1)}~, &  \;\;   (1-2\mu)^2 \geqslant gc \\ & \\  \displaystyle   2 \sqrt[4]{ \frac{ 1-gc}{ (g+c)^2 - 4(1-2\mu)^2}} \frac{ K(r_4)}{K(k_1)}      ~,& \;\;  (1-2\mu)^2 < gc \end{cases} \\
&R_P& =  \begin{cases} \displaystyle \sqrt{ \frac{ -g+c+2\sqrt{1-gc}}{ -g+c+2\sqrt{ (1-2\mu)^2-gc}}} \frac{ K(r_3)}{K(k_2)} ~,& \;\;\;   (1-2\mu)^2 \geqslant gc \\ & \\  \displaystyle   \sqrt{   \frac{ -g+c+2\sqrt{ 1-gc}}{ \sqrt{ (g+c)^2 - 4(1-2\mu)^2}}}\frac{ K(r_4)}{K(k_2)}      ~,&  \;\;\; (1-2\mu)^2 < gc \end{cases} 
\end{eqnarray*}
where the moduli $k_1, k_2, r_1, r_2, r_3 $, and $r_4$ are given as in Lemma \ref{lemma:periods}.
\end{Lemma}

It follows from Lemma \ref{lemma:periodsmonotone} that $R_c$ is strictly increasing in the region $L$. Dullin-Montgomery \cite[section 7]{syzygy} conjectured that it is strictly decreasing in $S' $-, $S $-, and $P $-regions. The next lemma shows that in the $S'$-region it   is strictly decreasing for $c<c_J$, which is needed to prove Theorem \ref{thm:index} in Section \ref{sec:maintheorem}.

\begin{Lemma} \label{monotoneSs} The rotation function $R_c = R_c(g)$ is strictly decreasing in the $S'$-region, provided that $c<c_J$.
\end{Lemma}
\begin{proof} Note that the $S'$-region appears for $\mu < 1/2$. Fix $ c<c_J$. It suffices to show that 
\begin{equation*}
\label{eq:slope}
\frac{ \partial \tau_{\eta}}{\partial g} \tau_{\xi} - \tau_{\eta} \frac{ \partial \tau_{\xi}}{\partial g} 
\end{equation*}
is negative for $ g \in [-c-2(1-2\mu), -c+2(1-2\mu)]$.  We introduce ${A_{\mu} = \sqrt{ (1-2\mu)^2-gc}}$ and ${B= g-c}$. Note that $B>0$ since we have assumed that $c<c_J$. Then the two periods become
\begin{equation*}
 \tau_{\xi}  = \frac{ K(k_1)}{\sqrt{2A_0}},  \;\;\;\;\;  \tau_{\eta} = \frac{ K(r_1)}{ \sqrt{2 A_{\mu}}} 
\end{equation*}
where
\begin{equation*}
 k_1^2 = \frac{1}{2} \begin{pmatrix} \displaystyle 1- \frac{ B}{2A_0}\end{pmatrix}, \;\;\;\;\; r_1^2 = \frac{1}{2} \begin{pmatrix} \displaystyle 1- \frac{ B}{2A_{\mu}}\end{pmatrix}.
\end{equation*}
Then we obtain 

\begin{eqnarray*}
\frac{ \partial \tau_{\eta}}{\partial g} \tau_{\xi} - \tau_{\eta} \frac{ \partial \tau_{\xi}}{\partial g} &=& \begin{pmatrix} \displaystyle \frac{ \partial \tau_{\eta} }{\partial A_{\mu}}\frac{\partial A_{\mu}}{ \partial g} + \frac{ \partial \tau_{\eta} }{\partial B}\frac{\partial B}{ \partial g} \end{pmatrix}  \tau_{\xi} -\tau_{\eta} \begin{pmatrix} \displaystyle  \frac{ \partial \tau_{\xi}}{\partial A_0 }  \frac{\partial A_{0}}{ \partial g} + \frac{ \partial \tau_{\xi}}{ \partial B} \frac{ \partial B }{ \partial g}\end{pmatrix}\\
&=& \begin{pmatrix} \displaystyle \frac{ \partial \tau_{\eta} }{\partial A_{\mu}}\frac{-\frac{1}{2}c}{ A_{\mu}} + \frac{ \partial \tau_{\eta} }{\partial B}  \end{pmatrix}  \tau_{\xi} -\tau_{\eta} \begin{pmatrix} \displaystyle  \frac{ \partial \tau_{\xi}}{\partial A_0 }  \frac{-\frac{1}{2}c}{ A_0} + \frac{ \partial \tau_{\xi}}{ \partial B} \end{pmatrix}\\
 &=& -\frac{c}{2} \begin{pmatrix} \displaystyle  \frac{ \partial \tau_{\eta} }{\partial A_{\mu}}\frac{ \tau_{\xi} }{ A_{\mu}}   -   \frac{ \partial \tau_{\xi}}{\partial A_0 }  \frac{ \tau_{\eta}}{ A_{0}} \end{pmatrix} + \begin{pmatrix} \displaystyle \frac{ \partial \tau_{\eta}}{ \partial B} \tau_{\xi} - \tau_{\eta} \frac{ \partial \tau_{\xi}}{ \partial B } \end{pmatrix} .
\end{eqnarray*}

We first claim that  $ \displaystyle \frac{ \partial \tau_{\eta}}{ \partial A_{\mu}} \frac{ \tau_{\xi}}{A_{\mu}} - \frac{\partial \tau_{\xi}}{\partial A_0} \frac{ \tau_{\eta}}{ A_0} < 0 $. To see this, we compute that
\begin{eqnarray*}
\frac{ \partial \tau_{\eta} }{\partial A_{\mu}}  &=&  \frac{\partial K}{\partial r_1^2}\frac{\partial r_1^2}{\partial A_{\mu}}\frac{1}{\sqrt{2 A_{\mu}} }   - \frac{K(r_1)}{\sqrt{2A_{\mu}}^{3}}\\
&=&\frac{ \partial K}{ \partial r_1^2} \frac{ B}{ 4 A_{\mu}^{2}}\frac{1}{\sqrt{2A_{\mu}}} - \frac{ K(r_1)}{ \sqrt{2 A_{\mu}}^{3}}\\
&=& \frac{1}{\sqrt{2A_{\mu}}^{3}}\bigg(   \frac{ \partial K}{ \partial r_1^2} \frac{ B}{ 2 A_{\mu}} - K(r_1) \bigg)\\
&=& \frac{1}{\sqrt{2 A_{\mu}}^{3}} \begin{pmatrix} \displaystyle \frac{ \partial K}{\partial r_1^2} (1-2r_1^2) - K(r_1) \end{pmatrix} \\
&=&- \frac{ \pi}{ 2 \sqrt{2A_{\mu}}^{3} } \sum_{n=0}^{\infty} \begin{pmatrix} \displaystyle \frac{ (2n-1) !!}{(2n)!!} \end{pmatrix}^2 \frac{ (2n+1)(2n+3)}{4(n+1)} r_1^{2n}  .
\end{eqnarray*}
Replacing $r_1$ with $k_1$ and setting $\mu=0$ give rise to
\begin{eqnarray*}
\frac{ \partial \tau_{\xi} }{\partial A_{0}}=- \frac{ \pi}{ 2 \sqrt{2A_{0}}^{3} } \sum_{n=0}^{\infty} \begin{pmatrix} \displaystyle \frac{ (2n-1) !!}{(2n)!!} \end{pmatrix}^2 \frac{ (2n+1)(2n+3)}{4(n+1)} k_1^{2n}  .
\end{eqnarray*}
It follows that
\begin{equation*}
 \frac{ \partial \tau_{\eta} }{\partial A_{\mu}}\frac{ \tau_{\xi} }{ A_{\mu}}   -   \frac{ \partial \tau_{\xi}}{\partial A_0 }  \frac{ \tau_{\eta}}{ A_{0}} =  - \frac{   \pi  K(k_1) K(r_1)}{ 8 \sqrt{  A_{\mu} A_0} }     \displaystyle \sum_{n=0}^{\infty} \begin{pmatrix} \displaystyle \frac{ (2n-1) !!}{(2n)!!} \end{pmatrix}^2 \frac{ (2n+1)(2n+3)}{4(n+1)} \bigg(\frac{r_1^{2n}}{K(r_1)A_{\mu}^2} -  \frac{k_1^{2n}}{ K(k_1)A_{0}^2}  \bigg).
\end{equation*}
To prove the claim, it suffices to show that  $ \frac{r_1^{2n}}{ K(r_1) A_{\mu}^2}$ is increasing in $\mu$. Indeed,

\begin{eqnarray*}
\frac{\partial}{\partial \mu}\frac{r_1^{2n}}{ K(r_1) A_{\mu}^2} &=& \frac{  n r_1^{2n-2} \frac{\partial r_1^2}{\partial \mu} K(r_1) A_{\mu}^2 - r_1^{2n} \bigg(   \frac{\partial K}{\partial r_1^2} \frac{\partial r_1^2}{\partial \mu} A_{\mu}^2 + 2K(r_1) A_{\mu} \frac{\partial A_{\mu}}{\partial \mu}     \bigg) }{ K(r_1) ^2 A_{\mu}^4} \\
&=& \frac{  \frac{\partial r_1^2}{\partial \mu} r_1^{2n-2} A_{\mu}^2 \bigg( n  K(r_1)  - r_1^{2 }     \frac{\partial K}{\partial r_1^2}  \bigg) - 2r_1^{2n} K(r_1)A_{\mu} \frac{ -2 (1-2\mu)}{A_{\mu}}      }{ K(r_1) ^2 A_{\mu}^4} \\
&=& \frac{  4 (1-2\mu) r_1^{2n}       }{ K(r_1)  A_{\mu}^4} >0.
\end{eqnarray*}
This proves the claim.

We next claim that $ \displaystyle \frac{ \partial \tau_{\eta}}{\partial B} \tau_{\xi} - \tau_{\eta} \frac{ \partial \tau_{\xi}}{\partial B } < 0   $ from which the lemma follows. In a similar way as above, we observe that 

\begin{equation*}
\frac{ \partial \tau_{\eta}}{\partial B} \tau_{\xi} - \tau_{\eta} \frac{ \partial \tau_{\xi}}{\partial B } =  - \frac{ K(r_1) K(k_1)}{ 8 \sqrt{ A_0 A_{\mu}}  } \begin{pmatrix}\displaystyle \frac{   \partial K / \partial r_1^2}{ K(r_1) A_{\mu}} - \frac{ \partial K / \partial k_1^2}{ K(k_1) A_0} \end{pmatrix}
\end{equation*}
To show that $\frac{   \partial K / \partial r_1^2}{ K(r_1) A_{\mu}}$ is increasing in $\mu$, we compute that
\begin{eqnarray*}
\frac{\partial}{\partial \mu}\frac{   \partial K / \partial r_1^2}{ K(r_1) A_{\mu}} &=& \frac{   K(r_1) \bigg( \frac{\partial^2 K}{\partial (r_1^2)^2 }\frac{\partial r_1^2}{\partial \mu}A_{\mu} - \frac{\partial K}{\partial r_1^2} \frac{\partial A_{\mu}}{\partial \mu} \bigg) - \bigg( \frac{\partial K}{\partial r_1^2} \bigg)^2 \frac{\partial r_1^2}{\partial \mu} A_{\mu} }{ K(r_1)^2 A_{\mu}^2}\\
&=& \frac{   \frac{ 2 (1-2\mu)K(r_1)}{A_{\mu}} \bigg(- \frac{B}{4A_{\mu}} \frac{\partial^2 K}{\partial (r_1^2)^2 } + \frac{\partial K}{\partial r_1^2}   \bigg)+ \bigg( \frac{\partial K}{\partial r_1^2} \bigg)^2 \frac{(1-2\mu)B}{2A_{\mu}^2}  }{ K(r_1)^2 A_{\mu}^2}.
\end{eqnarray*}
The claim then follows from
\begin{eqnarray*}
&& - \frac{B}{4A_{\mu}} \frac{\partial^2 K}{\partial (r_1^2)^2 } + \frac{\partial K}{\partial r_1^2}\\ &=& ( r_1^2 - \frac{1}{2} )\frac{\partial^2 K}{\partial (r_1^2)^2 } + \frac{\partial K}{\partial r_1^2} \\
&=& \frac{\pi}{2} \sum_{n =0}^{\infty} \bigg( \bigg( \frac{ (2n+1)!! }{(2n+2)!!}\bigg)^2  (n+1)n -         \bigg( \frac{ (2n+3)!! }{(2n+4)!!}\bigg)^2 \frac{(n+2)(n+1)}{2} + \bigg( \frac{ (2n+1)!! }{(2n+2)!!}\bigg)^2 (n+1)  \bigg)r_1^{2n}\\
&=& \frac{\pi}{2} \sum_{n =0}^{\infty} \bigg( \frac{ (2n+1)!!  }{(2n+2)!! }\bigg)^2  \bigg(    \frac{ (n+1)(4n^2 + 12n + 7)}{8(n+2)}  \bigg)r_1^{2n} >0.
\end{eqnarray*}
This completes the proof of the lemma.
\end{proof}

\begin{Remark}\label{ramkslope} \rm A similar computation shows that the $c$-derivative of the rotation function $R$ of the $S'$-region is positive. Hence along the curve $R_{S'}(g,c)=R_0$, the slopes $\partial c / \partial g$ are positive.
\end{Remark}

\begin{Corollary} 
\label{cor:exterior}
For each $c<c_J$, the rotation number of the exterior collision orbit in the Moon component is greater than that  in the Earth component:   $R_{\text{ext}}^E < R_{\text{ext}}^M$.
\end{Corollary}

The following lemma concerns the $S$-region.

\begin{Lemma}\label{lemmanocircle}
Given any point $(g_0, c_0 ) \in S$, there exists a smooth function $f=f_{(c_0, g_0)}: \R \rightarrow \R$ with $f(g_0)=c_0$ such that 
\begin{equation*}
\lim_{g\rightarrow g_0} \frac{\partial R_S(g,f(g))}{\partial g}  \neq 0.
\end{equation*}
\end{Lemma}
\begin{proof} Choose any $(g_0, c_0 ) \in S$. Let $f: \R \rightarrow \R$ be any smooth function such that $f(g_0)=c_0$ and
\begin{equation}\label{functionf}
f'(g_0) = \frac{ c_0^2 + g_0c_0 -2}{g_0^2 +  g_0c_0 -2}.
\end{equation}
Observe that the denominator is negative. Define
\begin{equation*}
A(g) :=\sqrt{       \frac{\sqrt{1-gf(g)}}{ g-f(g)+2\sqrt{ (1-2\mu)^2 -gf(g)}}} 
\end{equation*}
so that

\begin{equation*}
R_S(g, f(g)) = 2 A(g)  \frac{K(r_2)}{K(k_1)}.
\end{equation*}
We compute that
\begin{equation*}
\frac{\partial A}{\partial g} = \frac{1}{ 4A \sqrt{1-gf} ( g-f+2\sqrt{ (1-2\mu)^2-gf})^2} \bigg(         B_0(g) + \frac{8\mu(1-\mu)(f+gf')}{\sqrt{(1-2\mu)^2 -gf}}       \bigg),
\end{equation*}
where $B_{\mu}(g) = f(g)^2 + gf(g) - g^2 f'(g) - gf(g)f'(g) + 2(1-2\mu)^2f'(g) - 2(1-2\mu)^2 $.
On the other hand, we have
\begin{eqnarray*}
\frac{\partial k_1^2(g, f(g))}{\partial g}  &=& \frac{ B_0(g)}{8\sqrt{1-gf(g)}^3}\\
\frac{\partial r_2^2(g, f(g))}{\partial g}   &=& \frac{ -  2B_{\mu}(g)}{       ( g-f + 2\sqrt{ (1-2\mu)^2 -gf})^2\sqrt{ (1-2\mu)^2 -gf}                  }.
\end{eqnarray*}
It follows that
\begin{eqnarray*}
\label{eq1}
\frac{\partial R_S(g, f(g))}{\partial g} &=& \frac{\partial A(g)}{\partial g}  \frac{K(r_2)}{K(k_1)} + A(g) \frac{ \frac{\partial K(r_2)}{\partial r_2^2} \frac{\partial r_2^2}{\partial g} K(k_1) - K(r_2) \frac{\partial K(k_1)}{\partial k_1^2} \frac{\partial k_1^2}{\partial g} }{K(k_1)^2} \\
&=&   \square_1 \bigg(     2  K(k_1) \sqrt{1-gf}   \bigg( B_0 + \frac{ 8\mu(1-\mu)(f+gf')}{\sqrt{(1-2\mu)^2 -gf}}\bigg) \\
&&\hspace{0.7cm} -    B_0 \frac{\partial K}{\partial k_1^2}  \bigg( g-f +2\sqrt{ (1-2\mu)^2 -gf} \bigg)         \bigg) - \square_2 B_{\mu},
\end{eqnarray*}
where
\begin{eqnarray*}
\square_1 &=& \frac{K(r_2)}{ 8K(k_1)^2 A (1-gf)(g-f+2\sqrt{ (1-2\mu)^2 -gf})^2}  \\
\square_2 &=& \frac{ 2A}{K(k_1)  ( g-f + 2\sqrt{ (1-2\mu)^2 -gf})^2\sqrt{ (1-2\mu)^2 -gf} }\frac{\partial K(r_2)}{\partial r_2^2}
\end{eqnarray*}
are positive terms. In view of (\ref{functionf}) we obtain
\begin{equation*}
B_{\mu}(g_0) = \frac{ 8\mu(1-\mu)(g_0^2 -c_0^2 )}{g_0^2 + g_0c_0-2} \geqslant 0.
\end{equation*}
Moreover, we have

\begin{eqnarray*}
f(g_0) + g_0f'(g_0) &=& c_0 + g_0 \frac{ c_0^2 + g_0c_0 -2}{g_0^2 + g_0c_0 -2} \\
&=& \frac{ 2(g_0c_0-1)( g_0 +c_0)}{g_0^2 + g_0c_0-2} <0.
\end{eqnarray*}
It follows that
\begin{equation*}
\label{eq1}
\frac{\partial R_S(g_0, f(g_0))}{\partial g}= \bigg( \square_1  \frac{ 16\mu(1-\mu)K(k_1) \sqrt{1-gf} }{\sqrt{(1-2\mu)^2 -gf}}\bigg) (f+gf') -  \square_2 B_{\mu}<0.
\end{equation*}
This completes the proof of the lemma.
\end{proof}

\begin{Corollary} In the $S$-region, there is no critical point of the rotation function. 
\end{Corollary}

We are now in a position to prove the main result of this section.

\textit{Proof of Proposition \ref{prop1}.} 

 \;

\textbf{Case1}: The interior collision orbits.\\
 The interior collision orbits lie on  $l_3 : g=-c-2$ with $ c < c_J$. By Lemmas \ref{lemma:periods} and \ref{lemma:rotation},  the periods and the rotation function are given by
\begin{eqnarray*}
&& \tau_{\xi}^{\text{int}} = \frac{ \pi}{ 2\sqrt{2(-1-c)}}, \;\;\;\;\; \tau_{\eta}^{\text{int}} = \frac{ K(r_2)}{ \sqrt{  -c-1 + \sqrt{ c^2 + 2c+ (1-2\mu)^2        }}} \\
&& R_{\text{int}} = \frac{ 2\sqrt{2(-1-c)}K(r_2)}{ \pi \sqrt{ -c-1 + \sqrt{ c^2+2c+(1-2\mu)^2}}} = \frac{2}{\pi} \sqrt{ 1+ r_2^2}K(r_2),
\end{eqnarray*}
where 
$$ r_2^2 = \frac{ -c-1 -\sqrt{c^2+2c+(1-2\mu)^2} }{-c-1+\sqrt{c^2+2c+(1-2\mu)^2}}. $$
We compute that
\begin{eqnarray*}
\frac{\partial r_2^2}{\partial c} = \frac{ 8\mu(1-\mu)}{ ( -c-1+ A)^2 A} > 0,
\end{eqnarray*}
where $ A = \sqrt{ c^2 + 2c+ (1-2\mu)^2}.$ Then it follows from
\begin{eqnarray*}
 \frac{\partial R_{\text{int}}}{\partial c}  = \frac{1}{\pi \sqrt{1+r_2^2}}\frac{\partial r_2^2}{\partial c} \bigg( K(r_2) + 2(1+r_2^2) \frac{ \partial K}{\partial r_2^2} \bigg)
\end{eqnarray*}
that the rotation function $R_{\text{int}}$ is strictly increasing in $c$. Moreover, 
\begin{eqnarray*}
\lim_{c \rightarrow -\infty} R_{\text{int}} = \lim_{ r_2^2 \rightarrow 0} \frac{2}{\pi} \sqrt{ 1+r_2^2}K(r_2) = 1 , \;\;\;\;\; \lim_{c \rightarrow c_J} R_{\text{int}} = \lim_{ r_2^2 \rightarrow 1} \frac{2}{\pi} \sqrt{ 1+r_2^2}K(r_2) = \infty.
\end{eqnarray*}

\;\;

\textbf{Case2}: The exterior collision orbits in the Earth component.\\
  The exterior collision orbits in the Earth component lie on  $l_1 : g = -c + 2(1-2\mu)$ and we have
\begin{eqnarray*}
&&\tau_{\xi}^{\text{ext,E}} = \frac{  K(k_1)}{\sqrt{2} \sqrt[4]{ c^2 -2(1-2\mu)c +1}}, \;\;\;\;\; \tau_{\eta}^{\text{ext,E}} = \frac{ \pi}{2\sqrt{ 2(1-2\mu-c)}} \\
&&R_{\text{ext}}^E = \frac{ \pi \sqrt[4]{c^2 -2(1-2\mu)c +1}}{2\sqrt{1-2\mu-c}K(k_1)} = \frac{\pi}{2} \frac{ 1}{ \sqrt{ 1-2k_1^2}K(k_1)},
\end{eqnarray*}
where
\begin{equation*}
k_1^2 = \frac{1}{2} \bigg( 1 - \frac{ 1-2\mu-c}{\sqrt{ c^2 -2(1-2\mu)c + 1}}\bigg).
\end{equation*}
As in the previous case, we compute that
\begin{eqnarray*}
 \frac{ \partial k_1^2}{\partial c} &=& \frac{ 2 \mu(1-\mu)}{ \big( c^2 - 2(1-2\mu)c+1\big)^{3/2}} >0 \\
\frac{\partial R_{\text{ext}}^E}{\partial c} &=& \frac{\pi}{2 (1-2k_1^2)^{3/2} K(k_1)^2 } \frac{\partial k_1^2}{\partial c} \bigg( K(k_1) - (1-2k_1^2) \frac{\partial K}{\partial k_1^2} \bigg) \\
&=& \frac{\pi^2}{4 (1-2k_1^2)^{3/2} K(k_1)^2 } \frac{\partial k_1^2}{\partial c}  \bigg( \sum_{n=0}^{\infty}  \frac{ (2n-1)!! (2n+1)!!}{(2n)!!(2n+2)!!} \frac{ 2n+3}{2} k_1^{2n}\bigg) > 0.
\end{eqnarray*}
Moreover,

$$ \lim_{c \rightarrow -\infty} R_{\text{ext}}^E = \lim_{   k_1^2 \rightarrow 0 } \frac{\pi}{2} \frac{ 1}{ \sqrt{ 1-2k_1^2}K(k_1)} = 1$$
and
$$ \lim_{c \rightarrow 0 } R_{\text{ext}}^E = \lim_{k_1^2 \rightarrow \mu}  \frac{\pi}{2} \frac{ 1}{ \sqrt{ 1-2k_1^2}K(k_1)} = \frac{\pi}{2} \frac{1}{\sqrt{1-2\mu}}K(\sqrt{\mu}).$$

\;\; 

\textbf{Case3}: The exterior collision orbits in the Moon component.\\
 The exterior collision orbits in the Moon component lie on $l_2 : g=-c-2(1-2\mu)$ and we have
\begin{eqnarray*}
&&\tau_{\xi}^{\text{ext,M}} = \frac{ K(k_1)}{ \sqrt{2} \sqrt[4]{ c^2 + 2(1-2\mu)c + 1 }} , \;\;\;\;\;\tau_{\eta}^{\text{ext,M}} = \frac{ K(r_1)}{  \sqrt{2 |c+1-2\mu| }} \\
&&R_{\text{ext}}^M = \frac{  \sqrt[4]{c^2 + 2(1-2\mu)c +1}K(r_1)}{ \sqrt{|c+1-2\mu|} K(k_1)} = \frac{ K(r_1)}{  \sqrt{ |2k_1^2-1|}K(k_1)},
\end{eqnarray*}
where
$$ k_1 ^2 = \frac{1}{2} \bigg( 1 + \frac{ c+1-2\mu}{ \sqrt{ c^2 + 2(1-2\mu)c + 1}}\bigg), \;\;\;\;\; r_1^2 = \frac{1}{2} \bigg( 1 + \frac{ c+1-2\mu}{ |c+1-2\mu|}\bigg).$$
Assume first that $c< c_h = -1+2\mu$. Then we have
\begin{eqnarray*}
R_{\text{ext}}^M = \frac{ \pi}{2} \frac{ 1}{\sqrt{ 1-2k_1^2} K(k_1)}.
\end{eqnarray*}
Together with 
$$ \frac{\partial k_1^2}{\partial c } = \frac{ 2\mu(1-\mu)}{ \big( c^2 + 2(1-2\mu)c +1 \big)^{3/2}} >0,$$
the computation in the case2 shows that the derivative $\partial R_{\text{ext}}^M / \partial c $ is positive. Moreover, 
\begin{equation*}
\lim_{c \rightarrow -\infty} R_{\text{ext} }^M = \lim_{ k_1^2 \rightarrow 0}  \frac{\pi}{2} \frac{ 1}{ \sqrt{ 1-2k_1^2}K(k_1) } = 1.
\end{equation*}
If $c>c_h $, then $r_1^2 = 1$ and hence $R_{\text{ext}}^M = \infty$.

\;\;

\textbf{Case4}: The double-collision orbits.\\
The double-collision orbits lie on  $l_3 : g=-c-2$ with $ c > c_J$ and we have
\begin{eqnarray*}
&& \tau_{\xi}^{\text{dou}} = \begin{cases} \displaystyle \frac{\pi}{2\sqrt{2(-1-c)}} & c<c_e \\ & \\ \;\;\;\;\;\;\;\; \infty & c>c_e \end{cases}, \;\;\; \tau_{\eta}^{\text{dou}} =  \begin{cases} \displaystyle  \frac{K(r_4)}{\sqrt{2} \sqrt[4]{ \mu(1-\mu)}} & (1-2\mu)^2 < gc \\ & \\ \displaystyle \frac{  K(r_3)}{ \sqrt{ c+1 + \sqrt{ c^2 + 2c+(1-2\mu)^2}}} & (1-2\mu)^2 \geqslant gc \end{cases} \\
&& R_{\text{dou}} =    \begin{cases} \displaystyle  \frac{ 2\sqrt{-1-c} K(r_4)}{ \pi \sqrt[4]{\mu(1-\mu)} } = \frac{2}{\pi} \sqrt{ 4r_4^2-2} K(r_4)& (1-2\mu)^2 < gc,~ c<c_e \\ & \\ \displaystyle  \;\;\;\;\;\;\;\; 0 & \text{ otherwise } \end{cases}
\end{eqnarray*}
where

$$  r_3^2 = \frac{2\sqrt{ c^2 + 2c + (1-2\mu)^2}}{ c+1+\sqrt{c^2+2c+(1-2\mu)^2} }, \;\;\;\;\;  r_4^2 = \frac{1}{2} \bigg( 1 - \frac{ c+1}{ 2\sqrt{ \mu(1-\mu)} } \bigg). $$
We compute that for $(1-2\mu)^2 < gc$ with $c<c_e$
\begin{eqnarray*}
&&\frac{\partial r_4^2}{\partial c} = - \frac{1}{4\sqrt{\mu(1-\mu)}} < 0 ,\\
&&\frac{ \partial R_{\text{dou}}}{\partial c} = \frac{ 2\sqrt{2}}{\pi \sqrt{ 2r_4^2-1}}\frac{\partial r_4^2}{\partial c} \bigg( K(r_4) + (2r_4^2-1)\frac{\partial K}{\partial r_4^2} \bigg) < 0.
\end{eqnarray*}
Moreover,
\begin{eqnarray*}
\lim_{c \rightarrow c_J} R_{\text{dou}} = \lim_{ r_4^2 \rightarrow 1 } \frac{2}{\pi} \sqrt{4r_4^2 -2} K(r_4) = \infty, \;\;\;  \lim_{c \rightarrow c_e} R_{\text{dou}} = \lim_{ r_4^2 \rightarrow  1/2 } \frac{2}{\pi} \sqrt{4r_4^2 -2} K(r_4) = 0.
\end{eqnarray*}

\;

\textbf{Case5}: The hyperbolic orbits.\\
The hyperbolic  orbits lie on  $l_4 : gc= (1-2\mu)^2$ with $ c_J < c< c_h$ and we have
\begin{eqnarray*}
 \tau_{\xi}^{\text{hyp}} = \frac{ K(k_1)}{2 \sqrt[4]{\mu(1-\mu)} } , \;\;\;\;\; \tau_{\eta}^{\text{hyp}} = \infty, \;\;\;\;\; R_{\text{hyp}} = \infty,
\end{eqnarray*}
where
\begin{eqnarray*}
k_1^2 = \frac{1}{2} \bigg( 1 - \frac{(1-2\mu)^2 - c^2}{ 4c\sqrt{\mu(1-\mu)}} \bigg).
\end{eqnarray*}

\;

\textbf{Case6}: The elliptic orbits.\\
The elliptic  orbit lie on  $l_5 : gc=1$ with $c>c_e$ and we have
\begin{eqnarray*}
&&\tau_{\xi}^{\text{ell}} =  \frac{ \pi\sqrt{- c}}{\sqrt{2(1-c^2)}}, \;\;\;\;\; \tau_{\eta}^{\text{ell}} = \frac{ \sqrt{-2 c}K(r_4)}{\sqrt[4]{c^4 + 2c^2 - 4(1-2\mu)^2 c^2 +1}} \\
&&R_{\text{ell}} = \frac{2}{\pi} \sqrt{ 1-2r_4^2} K(r_4),
\end{eqnarray*} 
where
$$ r_4^2   = \frac{1}{2} \bigg( 1 + \frac{ c^2-1}{ \sqrt{ c^4 + 2c^2 + 1 -4(1-2\mu)^2 c^2}}\bigg).$$
We compute that
\begin{eqnarray*}
\frac{\partial r_4^2}{\partial c} = \frac{  8\mu(1-\mu)c(c^2+1) }{\big( c^4 + 2c^2 +1 - 4(1-2\mu)^2 c^2\big)^{ 3/2}} <0  .
\end{eqnarray*}
It then follows from
\begin{eqnarray*}
 \frac{\partial R_{\text{ell}}}{\partial c} &=& \frac{2}{\pi \sqrt{ 1-2r_4^2}} \frac{\partial r_4^2}{\partial c} \bigg( -K(r_4) + (1-2r_4^2) \frac{\partial K}{\partial r_4^2} \bigg) \\
&=& - \frac{1}{\sqrt{ 1-2r_4^2}} \frac{\partial r_4^2}{\partial c} \bigg( \sum_{n=0}^{\infty}  \frac{ (2n-1)!! (2n+1)!!}{(2n)!!(2n+2)!!} \frac{ 2n+3}{2} r_4^{2n}\bigg) >0
\end{eqnarray*}
that the rotation function $R_{\text{ell}}$ is strictly increasing in $c$. Moreover,
\begin{equation*}
\lim_{c \rightarrow 0 } R_{\text{ell}} = \lim_{r_4^2 \rightarrow 0} \frac{2\sqrt{ 1-2r_4^2 }K(r_4)}{\pi} =1, \;\;\;\;\; \lim_{c \rightarrow c_e} R_{\text{ell}}=\lim_{r_4^2 \rightarrow 1/2} \frac{2\sqrt{ 1-2r_4^2} K(r_4)}{\pi} =0.
\end{equation*}
This completes the proof of the proposition.

\begin{Corollary} 
\label{cor:bifurcation}
$(\text{a})$ in the region  $S'$, a $T_{k,l}$-torus family  converges to either $l_1$ or the $g$-axis  in one direction and   to  $l_2$  in the other direction.

$(\text{b})$ in the region $S$, it  converges to $l_2$ and $l_3$.

$(\text{c})$ in the region $L$, it converges to $l_3$ and the $g$-axis.

$(\text{d})$ in the region $P$, it  converges to  $l_5$ and the $g$-axis.
\end{Corollary}

In \cite[appendix B]{Contopoulos} Contopoulos proved by a direct calculation that in  ${S'}$ and ${S}$, it is always ${ \tau_{\eta} > \tau_{\xi} }$, namely $R \in (1, \infty) $. Geometrically, this means that there is no periodic orbit in both ${\Sigma_c^E}$ and ${\Sigma_c^M}$ which is closed after one revolution. Note that this result follows immediately from  Proposition \ref{prop1} and Corollary \ref{cor:bifurcation}. Moreover, we obtain 

\begin{Corollary}  In the region $P$, the rotation number is greater than zero and less than one: $R_P \in (0,1)$.  On the other hand, in the $L$-region we have $R_L \in (0, \infty).$
\end{Corollary}

Thus, for a $T_{k,l}$-torus family, we have $k>l$ for $S'$ and $S$ while $k<l$ for $P$. In $L$-region even $k=l$ is possible.

\begin{Corollary}
\label{cor:closed}
Suppose that on an energy hypersurface $H^{-1}(c)$  there exists a regular periodic orbit $\gamma$ which is closed after one revolution. Then the energy level is greater than the critical Jacobi energy, i.e., $ c > c_J$, and the associated value $(g,c)$ is contained in the region $L$.
\end{Corollary}

Corollary \ref{cor:bifurcation} implies that a ${T_{k,l}}$-type orbit is one of the five critical orbits at the extremal energies, which is abbreviated by  $c_{k,l}$. Let us see a transition of  periodic orbits in a $T_{k,l}$-torus family by increasing an energy level (cf. the transition in $T_{k,l}$-torus family in the rotating Kepler problem \cite[section 6]{RKP}). We start with $c<c_J$. A $T_{k,l}$-torus family is born out of a multiple cover of  the interior collision orbit at $c =c_{k,l}^{\text{int}}$. As the energy increases  periodic orbits become  regular. If this torus family represents the motions in the Moon component, then it  becomes a multiple cover of the exterior collision orbit $\gamma_{\text{ext}}^{\text{M}}$ in the Moon component at $c=c_{k,l}^{\text{ext,M}}$. Suppose that the torus family represents the motions in the Earth component. Since the rotation functions $R_{\text{int}}$ and $R_{\text{ext}}^{\text{E}}$  are strictly increasing, there exists the unique energy level $c_0 = c_0(\mu) < c_J$ such that a torus   family which bifurcates at $c \in (c_0, c_J)$(or $c<c_0)$ ends at the $g$-axis(or the line $l_1$).   Notice that for $\mu = 1/2$ this energy does not appear and any torus family which is born out of a multiple cover of the interior collision orbit ends with a multiple cover of the exterior collision orbit. We now consider the case $c_J < c < c_e$. A $T_{k,l}$-torus family is born out of a multiple cover of the double collision orbit at $c = c_{k,l}^{\text{dou}} $. As an energy increases, periodic orbits again become regular  and the family converges to the coordinate axis $c=0$. Finally, for $ c_e < c < 0$ a $T_{k,l}$-torus family is born out of a multiple cover of the elliptic orbit at $c=c_{k,l}^{\text{ell}}$ and the family behaves as the case with $c_J < c < c_e$, as the energy increases. Notice that no  Liouville torus bifurcates from the critical orbit whose rotation number is zero (the double collision orbit with $c_e < c< 0$) or infinite (the hyperbolic  orbit and the exterior collision orbit in the Moon component with $  c_h < c < 0$).  These critical orbits  are  unstable periodic orbits and the others    are stable, see  \cite[chapter 3]{Bifurcation}.

We conclude this section with providing  upper bounds of the rotation functions of the exterior collision orbits for energies below the critical Jacobi energy.

\begin{Lemma} \label{lemma:exteriorrotation}
 The rotation function $R_{\text{ext}}$ of the exterior collision orbit  in each component is less than $2$ for $c<c_J$. 
\end{Lemma}
\begin{proof} Fix ${\mu \in (0, 1/2]}$. By Corollary \ref{cor:exterior}, it suffices to prove the assertion for the Moon component. Moreover, since ${R_{\text{ext}}^M  }$ is increasing in $c$, it suffices to show $R_{\text{ext}}^M(c_J) <2$. Recall that
\begin{eqnarray*}
R_{\text{ext}}^M(c_J) = \frac{\pi}{2\sqrt{1-2k_1^2}K(k_1)} \leqslant \frac{1}{\sqrt{1-2k_1^2}},
\end{eqnarray*}
where

\begin{eqnarray*}
k_1^2 = \frac{1}{2} \bigg( 1 + \frac{ c_J + 1-2\mu}{\sqrt{ c_J^2 + 2(1-2\mu)c_J +1}} \bigg).
\end{eqnarray*}
We compute that
\begin{eqnarray*}
\frac{1}{\sqrt{1-2k_1^2}} < 2 \;\;\; &\Longleftrightarrow&  \;\;\; k_1^2 < \frac{3}{8} \\
&\Longleftrightarrow & \;\;\; \frac{\sqrt{ c_J^2 + 2(1-2\mu)c_J +1}}{4} < -c_J-(1-2\mu) \\
&\Longleftrightarrow & \;\;\; 64\mu(1-\mu) < 15 \big( c_J^2 + 2(1-2\mu)c_J +1 \big) .
\end{eqnarray*}
It then follows from $ (c_J+1)^2 = 4\mu(1-\mu)$ that
\begin{eqnarray*}
R_{\text{ext}}^M(c_J) < 2 \;\;\; &\Longleftrightarrow& \;\;\; 16(c_J+1)^2 < 15 \big( c_J^2 + 2(1-2\mu)c_J+1 \big) \\
&\Longleftrightarrow & \;\;\;  (c_J+1)^2 + 60\mu c_J <0 .
\end{eqnarray*}
We observe that 
\begin{eqnarray*}
 (c_J+1)^2 + 60\mu c_J &=&  4\mu(1-\mu) + 60\mu( -1-2\sqrt{\mu(1-\mu)} ) \\
&=& -4\mu^2 -56\mu -120\mu \sqrt{\mu(1-\mu)} <0.
\end{eqnarray*}
This completes the proof of the lemma. 
\end{proof}

\section{Main arguments}\label{sec:maintheorem}

\subsection{Definition of the Conley-Zehnder index} 
 
We recall the definition of the Robbin-Salamon index \cite{spectral}. Let  $\psi : [0,T] \rightarrow Sp(2n) $ be a smooth path of symplectic matrices. A point ${t \in [0,T] }$ is called {a crossing} if $\det(\psi(t) -  I ) =0$. For a crossing $t \in [0,T] $, {the crossing form} $Q_t$ is defined as the quadratic form 
\begin{equation*}
 Q_t(v, v) := \omega( v, \dot{\psi}(t)v), \;\;\;\;\; v \in \text{Eigen}_{1}(\psi(t)) ,
\end{equation*}
where $\text{Eigen}_1 (\psi(t)) $ is the eigenspace of $\psi(t)$ to the eigenvalue 1. A crossing $ t \in [0,T]$ is called {nondegenerate} if the crossing from $Q_t$ is nondegenerate, i.e., the corresponding matrix does not have  an eigenvalue  equal to zero. Assume that the path $\psi$ has only nondegenerate crossings. In particular, they are isolated. Then {the Robbin-Salamon index} is defined by
\begin{equation*}
\mu_{\text{RS}}(\psi) := \frac{1}{2} \text{sgn} (Q_0) + \sum_{ t \in (0,T),  \;  \text{crossing}} \text{sgn}(Q_t) + \frac{1}{2} \text{sgn}(Q_T).
\end{equation*}
If a path $\psi$ has degenerate crossing, we homotope $\psi$ to $\widetilde{\psi}$ with only nondegenerate crossings, where during the homotopy the  endpoints are fixed. Robbin-Salamon  proved that the Robbin-Salamon index is invariant under homotopy with fixing endpoints. Thus, we define

\begin{equation*}
\mu_{\text{RS} }(\psi ) := \mu_{\text{RS}} ( \widetilde{\psi}) .
\end{equation*}

Let $(M, \xi= \ker \alpha)$ be a contact three manifold with $\pi_2(M)=0$. Let $\gamma :S^1 \rightarrow M$ be a contractible Reeb orbit and let $u : D \rightarrow M$ be its spanning disk, i.e., $u(e^{2\pi  i t}) = \gamma(t)$. Choose a trivialization $\Psi : u^* \xi \rightarrow D \times \R^2$ of the pull-back bundle $u^* \xi \rightarrow D$. Then we obtain a path of symplectic matrices $\psi : [0,T] \rightarrow Sp(2)$, which is defined as
\begin{eqnarray*}
\psi(t) := \Psi(t) \left. \begin{matrix} d\phi_H^t(x) \end{matrix} \right|_{\xi} \Psi^{-1}(0),
\end{eqnarray*}
where $\Psi(t) = \Psi(e^{2 \pi i t})$, by linearizing the Hamiltonian flow along $\gamma$ with respect to the chosen trivialization. We define the Conley-Zehnder index of the periodic Reeb orbit $\gamma$ by
\begin{eqnarray*}
\mu_{\text{CZ} } (\gamma) := \mu_{\text{RS}}(\psi).
\end{eqnarray*}
Since $\pi_2(M)=0$ and $D$ is contractible, this definition is independent of the involved choices on the spanning disk and the trivialization.

\begin{Remark} \rm In this paper, we choose the definition of the Conley-Zehnder index as the one which is given  in \cite[chapter 3]{HWZ} which is lower-semicontinuous. It is well known that  this definition   coincides with the  above definition via the Robbin-Salamon index   for a nondegenerate periodic orbit, i.e., a periodic orbit $\gamma$ such that $\left. \begin{matrix} d\phi_H^T(x) \end{matrix} \right|_{\xi}$ does not have an eigenvalue equal to 1. Notice that the Robbin-Salamon index is not lower-semicontinuous and two definitions are different for degenerate orbits. 
\end{Remark}

\subsection{Proof of Theorem \ref{thm:formula}} In what follows, by an evenly-covered periodic orbit we mean a periodic orbit which is the 2N-fold iterate of some primitive periodic orbit, where N is a positive integer. We  compute the Conley-Zehnder indices of the evenly-covered  interior and  exterior collision orbits, which are contractible,  for energies below the critical Jacobi energy.   

We now introduce the doubly-covered elliptic coordinates  $(\lambda, \nu) \in \R \times [-\pi, \pi]$, which are defined by 
\begin{equation*}
q_1 = \frac{1}{2} \cosh \lambda \cos \nu \;\;\; \text{ and } \;\;\; q_2 = \frac{1}{2} \sinh \lambda \sin \nu 
\end{equation*}
and the momenta are determined by the relation $p_1 dq_1 + p_2 dq_2 = p_{\lambda} d\lambda + p_{\nu} d\nu$. Note that the new coordinates $(\lambda, \nu)$ are related to the single-covered elliptic coordinates $(\xi, \eta)$ by
\begin{equation*}
\label{eq:variables}
\xi = \cosh \lambda \;\;\; \text{ and } \;\;\;  \eta = \cos \nu
\end{equation*}
which imply that $\tau_{\lambda} = 2\tau_{\xi} $ and $\tau_{\nu} = 2 \tau_{\eta} $. The rotation number does not change. 
The Hamiltonian and the integral in this new coordinates are then given by
\begin{equation*}
H(\lambda, \nu, p_{\lambda}, p_{\nu}) = \frac{ H_{\lambda} + H_{\nu} }{ \cosh^2 \lambda - \cos^2 \nu } 
\end{equation*}
and
\begin{equation*}
G= -\frac{ H_{\lambda} \cos^2 \nu + H_{\nu} \cosh^2 \lambda }{\cosh^2 \lambda - \cos^2 \nu } ,
\end{equation*}
where
\begin{equation*}
H_{\lambda} = 2p_{\lambda}^2 - 2 \cosh \lambda\;\;\; \text{ and } \;\;\;H_{\nu} = 2p_{\nu}^2 + 2(1-2\mu) \cos \nu .
\end{equation*}
We define the regularized Hamiltonian $K$ by
\begin{equation*}
K := K_c = (H-c)( \cosh^2 \lambda - \cos^2 \nu ) = K_{\lambda} + K_{\nu},
\end{equation*}
where 
\begin{equation*}
K_{\lambda } = 2p_{\lambda}^2 - 2 \cosh\lambda -c \cosh^2 \lambda \;\;\; \text{ and } \;\;\; K_{\nu} = 2p_{\nu}^2 + 2(1-2\mu)\cos \nu + c \cos^2 \nu.
\end{equation*}
 With the time scaling $dt = (\cosh^2 \lambda- \cos^2 \nu) d\tau$, we examine orbits of $K$ with energy 0 and time parameter $\tau$. The regularized Hamiltonian vector field  is given by
\begin{equation*}
X_K = 4p_{\lambda} \partial_{\lambda} + 4 p_{\nu} \partial_{\nu} + 2\sinh \lambda ( 1 + c \cosh \lambda) \partial_{p_{\lambda}} + 2 \sin \nu ( 1-2\mu  + c \cos \nu) \partial _{p_{\nu}}.
\end{equation*}

\begin{Remark} \rm For $c<c_J$, the regularized energy hypersurface $K^{-1}_c(0)$ is diffeomorphic to the disjoint union of two three-spheres, which is the double cover(or the universal cover) of $\R P^3$. For $c>c_J$, it is diffeomorphic to $S^2 \times S^1$, which is the double cover of the connected sum $\R P^3 \sharp \R P^3$. 
\end{Remark}

\textbf{Case1:} \textit{The interior collision orbits}. 

 The interior collision orbit  has constant $\lambda$, and hence $p_{\lambda} = 0 $. It follows that   the Hamiltonian flow along the interior collision orbit   is given by
\begin{equation*}
X_K   =   4p_{\nu}\partial_{\nu} +   2\sin \nu (  1-2\mu  + c \cos \nu )\partial_{p_{\nu}} .
\end{equation*}
In particular, the period  of the regularized orbit equals the $\nu$-period. On the other hand, the two linearly independent vector fields $\partial_{\lambda}$ and $\partial_{p_{\lambda}}$ lie in $\text{ker} \alpha \cap \text{ker}dK$, where $\alpha$ is the canonical 1-form.  In other words, they  trivialize the contact structures along the interior collision orbits.  With respect to this trivialization,  the linearized Hamiltonian flow along $\gamma_{\text{int}}$ onto the contact structure is a solution of the ODE
\begin{equation*}
\label{eq:interiorODE}
\dot{\psi}_{\text{int}} = \begin{pmatrix} 0 & 4 \\  2( 1+c) & 0 \end{pmatrix} \psi_{\text{int}},
\end{equation*}  
where $ 2(1+c)<0$ in view of $c<c_J$.  Solving the ODE yields the map
\begin{equation*}
\psi_{\text{int}}(\tau) = \begin{pmatrix} \cos  2\sqrt{-2( 1+c)}  \tau & \sqrt{ \frac{-2}{ 1+c}} \sin   2\sqrt{ -2( 1+c)}   \tau \\ - \sqrt{ \frac{  1+c}{-2}} \sin  2\sqrt{-2( 1+c)}  \tau &  \cos   2\sqrt{ -2( 1+c)}   \tau \end{pmatrix}.
\end{equation*} 
The crossings occur at 
\begin{equation*}
\tau \in \frac{ \pi}{\sqrt{-2( 1+c)} } \Z =   \tau_{\lambda}^{\text{int}} \Z.
\end{equation*}
and the crossing forms have signature 2.  Denote by $\gamma_{\text{int}}^{2N}$ the $2N$-th iteration of $\gamma_{\text{int}}$. The argument so far implies that it fails to be  nondegenerate if and only if $2N \tau_{\nu}^{\text{int}} \in  \tau_{\lambda}^{\text{int}} \Z$, or equivalently
\begin{equation*}
2N R _{\text{int}} \in \Z ,
\end{equation*}
and for a nondegenerate orbit the Conley-Zehnder index is then given by
\begin{equation}
\label{eq:interior}
\mu_{\text{CZ}}( \gamma_{\text{int}}^{2N}) = 1 + 2 \max \left \{ k \in \Z : k < 2N R_{\text{int}} \right \}.
\end{equation}

\;\;\;
\textbf{Case2:} \textit{The exterior collision orbits}. 

We proceed as in the previous case. Since the exterior collision orbits have constant $\nu$, we have $p_{\nu}=0$. Therefore,   the Hamiltonian flow along the exterior collision orbit   is given by
\begin{equation*}
X_K   =     4p_{\lambda} \partial_{\lambda}  + 2\sinh \lambda ( 1 + c \cosh \lambda) \partial_{p_{\lambda}} .
\end{equation*}
The period  of the regularized orbit  equals the  $\lambda$-period. The two vector fields ${\partial_{\nu}}$ and $\partial_{p_{\nu}}$ lie in $\text{ker}\alpha \cap \text{ker}dK$ and this implies that they trivialize the contact structures along the orbits.

 Following the similar computation as in the previous case, we obtain the associated ODE 
\begin{equation*}
\label{eq:ODE1}
\dot{\psi}_{\text{ext}}^{\pm} = \begin{pmatrix} 0 & 4 \\ - 2(  \pm(1-2\mu)-c) & 0 \end{pmatrix} {\psi}_{\text{ext}}^{\pm}
\end{equation*} 
for the Earth($+$) and the Moon($-$) components, respectively. Since $\pm(1-2\mu)-c>0$,  the above ODE   yields the map
\begin{equation*} 
\psi_{\text{ext}}^{\pm}(\tau) = \begin{pmatrix} \cos 2\sqrt{ 2(\pm( 1-2\mu)  -c)}\tau & \sqrt{ \frac{ 2 }{ \pm( 1-2\mu)  - c}} \sin 2\sqrt{ 2( \pm(1-2\mu)  -c)}\tau \\ - \sqrt{ \frac{ \pm( 1-2\mu) -c}{2}} \sin 2\sqrt{ 2(\pm( 1-2\mu)  -c)}\tau & \cos 2\sqrt{ 2(\pm( 1-2\mu ) -c)}\tau \end{pmatrix}.
\end{equation*} 
This path of symplectic matrices has crossings at 
\begin{equation*}
\tau \in \frac{ \pi}{ \sqrt{2( \pm( 1-2\mu)  -c )}} \Z = \tau_{\nu}^{\text{ext},\pm} \Z 
\end{equation*} 
and the crossing form has signature 2. 
It follows that the $2N$-covered  exterior collision orbit fails to be nondegenerate if and only if 
\begin{equation*}
2N / R_{\text{ext}}^{\pm} \in \Z,
\end{equation*}
and for a nondegenerate exterior collision orbit  the Conley-Zehnder index is given by
\begin{equation}\label{eq:extcz}
\mu_{\text{CZ}}(\gamma_{\text{ext}}^{\pm, 2N} ) = 1 + 2 \max \left \{ k \in \Z : k < 2N/R_{\text{ext}}^{\pm}  \right \} .
\end{equation} 

   \;\;\;

\;\;\;

\;\;\;

We now consider the doubly-covered collision orbits. By Proposition \ref{prop1} and the formula (\ref{eq:interior}) the doubly-covered interior collision orbit is nondegenerate if $R \neq k/2$, $k \geqslant 3$, and for $R \in ( (k-1)/2, k/2)$ the Conley-Zehnder index is given by
\begin{equation*}
\mu_{\text{CZ}}(\gamma_{\text{int}}^{2} ) =    1 + 2 \max \left \{ N \in \Z : N < k  \right \} = 2k-1.
\end{equation*}

For the doubly-covered exterior collision orbit,  Lemma \ref{lemma:exteriorrotation}  and the formula (\ref{eq:extcz}) imply that it is always nondegenerate and  the Conley-Zehnder index equals 3. This completes the proof of the theorem.

\subsection{Proof of Theorem \ref{thm:index}} In \cite[theorem 1.1]{Kim2}, it is shown that the doubly-covered elliptic coordinates provide a 2-to-1 symplectic embedding into $\R^4$ having the property that the image of the regularized Moon component is convex for $c<c_J$. Then by \cite[theorem 3.7]{HWZ} the regularized Moon component is dynamically convex. 

To show that the regularized Earth component is also dynamically convex, we observe that by Remark \ref{ramket}, Lemma \ref{monotoneSs} and Remark  \ref{ramkslope} the rotation function $R_c=R_c(g)$ is strictly decreasing for $g \in I_{\epsilon}:=(-c-2(1-2\mu)-\epsilon, -c+2(1-2\mu))$, where   $\epsilon>0$ small enough, so that the $T_{k,l}$-torus family is Morse-Bott nondegenerate for $g \in I_{\epsilon}$. By homotopy invariance of the Conley-Zehnder index that any two periodic orbits in the $T_{k,l}$-torus family for $g \in I_{\epsilon}$ have the same Conley-Zehnder index.  Choose any $T_{k,l}$-type orbit associated with $(g,c) \in S$. Since this point represents periodic orbits both in the Earth and Moon components, periodic orbits  associated with $(g,c) \in S$ in the Earth component have the same Conley-Zehnder index as such orbits in the Moon components, i.e., the Conley-Zehnder index is greater than or equal to 3.  We then conclude from dynamical convexity of the Moon component and Morse-Bott nondegeneracy of the torus families for $g \in I_{\epsilon}$  that any (evenly-covered) torus type orbit in the $S'$-region is also of Conley-Zehnder index greater than or equal to 3. 

Albers-Fish-Frauenfelder-van Koert proved that in the rotating Kepler problem the circular orbits are contractible if and only if they are evenly-covered , see \cite[claim in section 7.2]{RKP}. Their argument also holds for the Euler problem, i.e., the interior and the exterior collision orbits are contractible if and only if they are evenly-covered. In particular, by Theorem \ref{thm:formula} the doubly-covered exterior collision orbit has the Conley-Zehnder index 3. This implies that  the regularized Earth component is dynamically convex. This finishes the proof of the theorem.

\begin{Remark} \rm Dullin-Montgomery conjecture for the $S$-region which is mentioned in the previous section has the following implication: since the rotation function is strictly decreasing also in the $S$-region, the whole $T_{k,l}$-torus family is Morse-Bott nondegenerate. Therefore, to determine the Conley-Zehnder index of a torus type orbit   it suffices to determine the index of one periodic orbit in each torus family. Then using invariance of the local Floer homology (see  \cite{Kai}) as in \cite[chapter 7]{RKP} one can determine that  the Conley-Zehnder index of the $T_{k,l}$-torus family equals $2k-1$. Note that since the covering number of the $T_{k,l}$-torus family is $l$, to be the orbits evenly-covered, $l$ must be greater than 2. It follows from $k>l$ that $k \geqslant 3$ and then the Conley-Zehnder indices of any contractible torus type orbits are greater than 5. In particular, the doubly-covered exterior collision orbit is the unique periodic orbit of Conley-Zehnder index 3 on each compact component.
\end{Remark}

\appendix

\section{Computations}
\label{appendix:lemma}
In this appendix, we prove Lemma \ref{lemma:periodsmonotone}. It suffices to compute with the integrals (\ref{integralxi}) and (\ref{integraleta}). We denote them by $I_{\xi} $ and $I_{\eta}$, respectively. For convenience,  we introduce two functions ${A_{\mu} =  (1-2\mu)^2-gc }$ and ${B= g-c}$.  We first compute that 

\begin{eqnarray*}
\frac{ \partial I_{\eta}^{S'}}{ \partial g} &=& \frac{1}{ 4 A_{\mu}^{5/4}} \bigg( -\sqrt{A_{\mu}} \frac{\partial K}{\partial r_1^2} + c\big( K(r_1) - \frac{B}{2\sqrt{A_{\mu}}} \frac{\partial K}{\partial r_1^2} \big) \bigg) \\
&=&\frac{1}{ 4 A_{\mu}^{5/4}} \bigg( -\sqrt{A_{\mu}} \frac{\partial K}{\partial r_1^2} + c\big( K(r_1) + ( 2r_1 ^2 -1) \frac{\partial K}{\partial r_1^2} \big) \bigg) .
\end{eqnarray*}
Since 
\begin{eqnarray}
\label{eq:compare}
  K(r_1) + (2r_1^2-1)\frac{ \partial K}{\partial r_1^2}   =   \frac{\pi   }{2} \sum_{n=0}^{\infty} \begin{pmatrix} \displaystyle \frac{ (2n-1) !!}{(2n)!!} \end{pmatrix}^2 \frac{ (2n+1)(2n+3)}{4(n+1)} r_1^{2n}  > 0,
\end{eqnarray}
we conclude that ${\partial I_{\eta}^{S'} / \partial g < 0 }$. Plugging $\mu = 0$ gives rise to the same result for $I_{\xi}^{S'} = I_{\xi}^{S }=I_{\xi}^L$.

We now consider $  I_{\eta}^S   $. We compute that

\begin{eqnarray*}
\frac{\partial  I_{\eta}^S}{\partial g} &=& \frac{ 2 }{ (B +2 \sqrt{A_{\mu}})^{3/2} } \bigg( \big( B + 2\sqrt{A_{\mu}} \big) \frac{\partial K}{\partial r_2^2} \frac{\partial r_2^2}{\partial g} -  \big( 1 - \frac{c}{\sqrt{A_{\mu}} } \big)\frac{ K(r_2)}{2}  \bigg) \\
&=& \frac{ 2 }{ (B +2 \sqrt{A_{\mu}})^{3/2} } \bigg( \frac{ 2cB +4A_{\mu}}{ (B+2\sqrt{A_{\mu}})\sqrt{A_{\mu} }}  \frac{\partial K}{\partial r_2^2}   -  \big( 1 - \frac{c}{\sqrt{A_{\mu}} } \big)\frac{ K(r_2)}{2}  \bigg) \\
&=& \frac{ 2 }{ (B +2 \sqrt{A_{\mu}})^{3/2} } \bigg( \frac{1}{2} \big( \frac{ 8\sqrt{A_{\mu}}}{ B+2\sqrt{A_{\mu}}} \frac{\partial K}{\partial r_2^2} - K(r_2) \big) + \frac{c}{2 \sqrt{A_{\mu}}} \big( \frac{ 4B}{ B+2\sqrt{A_{\mu}}} \frac{\partial K}{\partial r_2^2} + K(r_2) \big) \bigg)\\
&=& \frac{ 2 }{ (B +2 \sqrt{A_{\mu}})^{3/2} } \bigg( \frac{1}{2} \big( (2- 2r_2^2)\frac{\partial K}{\partial r_2^2} - K(r_2) \big) + \frac{c}{2 \sqrt{A_{\mu}}} \big( \frac{ 4B}{ B+2\sqrt{A_{\mu}}} \frac{\partial K}{\partial r_2^2} + K(r_2) \big) \bigg).
\end{eqnarray*}
Since 
\begin{eqnarray*}
(2-2r_2^2)\frac{\partial K}{\partial r_2^2} - K(r_2) = - \frac{\pi}{2} \sum_{n=0}^{\infty} \frac{ (2n-1)!! (2n+1)!!}{ (2n)!! (2n+2)!!} r_2^{2n} <0,
\end{eqnarray*}
we conclude that the derivative $\partial  I_{\eta}^S/ \partial g$ is negative.

Consider the $\eta$-period in the region $L$ or $P$ with $(1-2\mu)^2 \geqslant gc$. We compute that
\begin{eqnarray*}
\frac{\partial I_{\eta}}{\partial g} &=& \frac{2}{\big( -B + 2\sqrt{A_{\mu} } \big)^{3/2}} \bigg( \frac{ 4A_{\mu} +2cB}{ \sqrt{A_{\mu}} ( -B+2\sqrt{A_{\mu}} ) } \frac{ \partial K}{\partial r_3^2} + \big( 1 + \frac{c}{\sqrt{A_{\mu}}} \big) \frac{K(r_3)}{2}  \bigg) \\
&=&\frac{2  }{ \big( -B+2 \sqrt{A_{\mu}}  \big)^{3/2}} \bigg(   \frac{c}{2\sqrt{A_{\mu}}} \big(  \frac{ 4B}{ -B+2\sqrt{A_{\mu}} }\frac{\partial K}{\partial r_3^2} + K(r_3) \big) + \frac{ 4\sqrt{A_{\mu}}}{-B+ 2\sqrt{A_{\mu}} } \frac{\partial K}{\partial r_3^2} + \frac{K(r_3)}{2} \bigg)\\
&=&\frac{2  }{ \big( -B+2 \sqrt{A_{\mu}}  \big)^{3/2}} \bigg(   \frac{c}{2\sqrt{A_{\mu}}} \big(  (2r_3^2-4)\frac{\partial K}{\partial r_3^2} + K(r_3) \big) + \frac{ 4\sqrt{A_{\mu}}}{-B+ 2\sqrt{A_{\mu}} } \frac{\partial K}{\partial r_3^2} + \frac{K(r_3)}{2} \bigg).
\end{eqnarray*}
Observe that
\begin{eqnarray*}
(2r_3^2 -4)\frac{ \partial K}{\partial r_3^2} + K(r_3) = -\frac{\pi}{2} \sum_{n=0}^{\infty}   \frac{ (2n-1)!! (2n+1)!!}{ (2n-2)!!(2n+2)!!}   r_3^{2n} < 0 .
\end{eqnarray*}
This implies that $\partial I_{\eta} / \partial g> 0 $. Plugging $\mu=0$ gives rise to the same result for $I_{\xi}^P$.

It remains to check the $\eta$-period for the region $L$ or ${P}$ with $(1-2\mu)^2 < gc$. Similarly, we compute that

\begin{eqnarray*}
 \frac{\partial I_{\eta}}{\partial g} &=& \frac{ 2 }{ \big( B^2 - 4A_{\mu} \big)^{5/4} } \bigg ( \frac{ \partial K}{\partial r_4^2} \frac{\partial r_4^2}{\partial g} \big(B^2-4A_{\mu} \big) - \frac{ B+2c}{2} K(r_4) \bigg)\\
&=& -\frac{ 2 }{ \big( B^2 - 4A_{\mu} \big)^{5/4} } \bigg ( \frac{2A_{\mu}+cB}{\sqrt{ B^2-4A_{\mu}}}        \frac{ \partial K}{\partial r_4^2}  +  \frac{ B+2c}{2} K(r_4) \bigg)\\
&=& -\frac{ 2 }{ \big( B^2 - 4A_{\mu} \big)^{5/4} } \bigg ( \frac{2A_{\mu} }{\sqrt{ B^2-4A_{\mu}}}        \frac{ \partial K}{\partial r_4^2}  +  \frac{ B }{2} K(r_4) + c \big( \frac{  B}{\sqrt{ B^2-4A_{\mu}}}        \frac{ \partial K}{\partial r_4^2}  +    K(r_4) \big)  \bigg)\\
&=& -\frac{ 2 }{ \big( B^2 - 4A_{\mu} \big)^{5/4} } \bigg ( \frac{2A_{\mu} }{\sqrt{ B^2-4A_{\mu}}}        \frac{ \partial K}{\partial r_4^2}  +  \frac{ B }{2} K(r_4) + c \big( (2r_4^2-1)      \frac{ \partial K}{\partial r_4^2}  +    K(r_4) \big)  \bigg).
\end{eqnarray*}
Together with (\ref{eq:compare}) it follows that $\partial I_{\eta} / \partial g > 0$ if $B<0$. We now assume that $B>0$.  Since $A_{\mu}<0$ and  $B+2c = g+c <0$, the second equality  implies that $\partial I_{\eta} / \partial g > 0$. This completes the proof of the lemma.

\section{Contact structures on energy hypersurfaces}

A hypersurface $\Sigma$ in a symplectic manifold $(M, \omega)$ is said to be {of restricted contact type} if there exists a Liouville vector field $Y$ on $M$, i.e., $\mathcal{L}_Y \omega = \omega$, which is transverse to $\Sigma$. In this case, the 1-form $\lambda := \iota_Y \omega = \omega(Y, \cdot)$ defines a contact form on $\Sigma$. 

In this appendix we show that energy hypersurfaces in the Euler problem for $c<c_J$ are of restricted contact type. The argument is  a slight modification of that in \cite[chapter 5]{contact}(in fact, it is much simpler). We repeat that here just for completeness.

Fix $c \in (-\infty, c_J)$. Consider the following Liouville vector field
\begin{equation*}
X= ( q_1+1/2)  \partial_{q_1} + q_2 \partial_{q_2} =: (q-E) \partial_q.
\end{equation*} 
We will show that the vector field $X$ intersects $H^{-1}(c)$ transversally for any $\mu \in (0,1)$.  Without loss of generality we may assume that $q_2 \geqslant 0$. We introduce the Earth polar coordinates $(r , \theta) \in (0, \infty) \times  [0, \pi]$, i.e., $(q_1, q_2) - E = ( r \cos \theta, r \sin \theta).$  Then the Hamiltonian and the vector field $X$ become
\begin{equation*} \label{polarhamiltonian}
H(r, \theta, p_1, p_2)=  \frac{1}{2}|p|^2 + V(r, \theta)=\frac{1}{2}|p|^2 - \frac{1-\mu}{r} - \frac{\mu}{ \sqrt{   r^2 - 2r \cos \theta +1  }}
\end{equation*}
and
\begin{equation*}\label{polarvectorfield}
X= r \partial_r .
\end{equation*}
Then it suffices to  show that
\begin{equation*}
\left. \begin{matrix} X(H) \end{matrix} \right|_{H^{-1}(c)} = \left. \begin{matrix} \displaystyle r \frac{ \partial V}{\partial r} \end{matrix} \right|_{ H^{-1}(c)} >0, 
\end{equation*} 
which is equivalent to 
\begin{equation} \label{eq:derofV}
\left. \begin{matrix} \displaystyle   \frac{ \partial V}{\partial r} \end{matrix} \right|_{  H^{-1}(c)}   = \left. \begin{matrix} \bigg( \displaystyle  \frac{1-\mu}{r^2} + \frac{ \mu(r- \cos \theta)}{ \sqrt{  r^2 - 2r \cos \theta +1}^3} \bigg)  \end{matrix} \right|_{H^{-1}(c)} >0. 
\end{equation}
Recall that $r =l$ is the unique root of the equation
\begin{equation*}
\frac{\partial V(r,0)}{\partial r}= \frac{1-\mu}{r^2} - \frac{ \mu}{(1-r)^2} .
\end{equation*}

\;

\textbf{Claim}. Given $r<1$, the derivative $\partial V/\partial r$ attains the  global minimum at $\theta =0$. \\
  For each $r$, we set
\begin{equation*} \label{eqU}
U_r( \theta):= \frac{\partial V(r, \theta) }{\partial r} = \frac{1-\mu}{r^2} + \frac{ \mu(r- \cos \theta)}{ \sqrt{ r^2 - 2r \cos \theta +1}^3}, \;\;\; \theta \in S^1
\end{equation*}
and compute that
\begin{eqnarray*}
\frac{ \partial U_r }{\partial \theta} &=& \frac{  \mu \sin \theta ( - 2 r^2 + r \cos \theta +1)}{\sqrt{ r^2 - 2r \cos \theta +1}^5} \\
\frac{ \partial^2 U_r }{\partial \theta^2} &=& \frac{  \mu \cos \theta ( - 2 r^2 + r \cos \theta +1)}{\sqrt{ r^2 - 2r \cos \theta +1}^5}  - \frac{  \mu r \sin^2 \theta }{\sqrt{ r^2 - 2r \cos \theta +1}^5} -\frac{ 5 \mu r \sin ^2 \theta ( - 2 r^2 + r \cos \theta +1)}{\sqrt{ r^2 - 2r \cos \theta +1}^7} .
\end{eqnarray*}
We first observe that $\theta=0$ and $\theta=\pi$ are critical points. For a given $r$ we have $-2r^2 + r \cos \theta_0 +1 = 0 $ if and only if
\begin{equation*}
\cos \theta_0 = \frac{ 2r^2 -1}{r}.
\end{equation*}
We see that $-1 \leqslant  (2r^2-1)/r \leqslant 1 $ implies that $r \geqslant 1/2$. In particular, $r=1/2$ if and only if $\cos \theta_0 = -1$, i.e.,  $\theta_0 = \pi$. Summarizing, for $r \leqslant 1/2$, there exist precisely two critical points of $U_r$, i.e., $\theta = 0$ and $\theta= \pi$, while for $r >1/2$ there exists another critical point $\theta_0 = \theta_0(r)$ satisfying $-2r^2 + r \cos \theta_0 +1 =0$.  Assume that $r > 1/2$.  Then we have
\begin{equation*}
\frac{ \partial^2 U_r (\theta_0) }{\partial \theta^2} = \frac{ \partial^2 U_r (2\pi - \theta_0) }{\partial \theta^2}=   - \frac{  \mu r \sin^2 \theta_0 }{\sqrt{ r^2 - 2r \cos \theta_0 +1}^5}  <0.
\end{equation*}
This implies that $\theta=\theta_0$ are a  local maximum. 

On the other hand, we compute that
\begin{eqnarray*}
\frac{\partial^2 U_r(0)}{\partial \theta^2}  &=& \frac{\mu ( -2r^2 + r + 1)}{ |r-1|^5} >0 \\
\frac{\partial^2 U_r(\pi)}{\partial \theta^2}  &=& \frac{\mu ( 2r^2 + r - 1)}{ |r+1|^5} \begin{cases}   <0 & \text{ if } r<\frac{1}{2} \\     =0 & \text{ if } r=\frac{1}{2}    \\  >0 & \text{ if } r>\frac{1}{2}      \end{cases}
\end{eqnarray*}
and
\begin{equation*}
U_r (0) - U_r(\pi) = -\frac{ \mu }{(1-r)^2}  -\frac{ \mu }{(1+r)^2} <0 .
\end{equation*}
This implies that $\theta=0$ is the unique global minimum for $U_r$, $0< r < 1$.  This proves the claim.

Now the same argument in the proof of \cite[corollary 5.3]{contact} shows that 
\begin{equation*}
\mathcal{K}_c^E \subset \left \{     (q_1, q_2) : r < l       \right \}. 
\end{equation*}
Since  $\partial V(r,0) / \partial r > 0 $  for $r <l$, together with the previous claim this shows that (\ref{eq:derofV}) holds true.

\begin{Remark} \rm One can perform Moser's regularization \cite[section 2]{Moser} to the Euler problem. To extend the vector field $X$ to the regularization, we only need to consider $|q|<\epsilon$.  Performing the coordinates changes in Moser's regularization we see that
\begin{equation*}
|q|= |\eta|(1-\xi_0 ) < \epsilon
\end{equation*}
and
\begin{equation*}
X =  \eta_0 \partial_{\eta_0} + \eta_1 \partial_{\eta_1} + \eta_2 \partial_{\eta_2},
\end{equation*}
where $(\xi, \eta) \in T^* S^2$. 
One can check that the Liouville vector field $ X=\eta \partial _{\eta}$ is transverse to the regularized energy hypersurface with $|\eta|(1-\xi_0) <\epsilon$ for sufficiently small $\epsilon >0$, for example, see \cite[section 6.2]{contact}. This means that we have a global Liouville vector field which is transverse to energy hypersurfaces for energies below the critical Jacobi energy. Therefore, for any $c<c_J$ the regularized energy hypersurface is of restricted contact type. 
\end{Remark}

We have proven the following.
\begin{Proposition} \rm  For $c<c_J$, each connected component of the energy hypersurface can be regularized to form the three-dimensional manifold which is diffeomorphic to $\R P^3 $. The two connected components of the regularized energy hypersurface are fiberwise star-shaped or of restricted contact type, where the transverse Liouville vector fields are given by
\begin{equation*}
X_E = (q-E) { \partial _q} \;\;\;\;\; \text{and} \;\;\;\;\;X_M = (q-M){\partial _q} .
\end{equation*}
\end{Proposition}

\;\;\;

\end{document}